\documentclass[a4paper,12pt]{article}

\usepackage[utf8]{inputenc}
\usepackage{amsthm,amsmath,stmaryrd,bbm,hyperref,geometry,color,xcolor}
\usepackage{amssymb}
\usepackage[english]{babel}
\usepackage{graphicx}
\usepackage{amsfonts,amssymb}
\usepackage{verbatim}
\usepackage{enumitem}
\usepackage[all]{xy}
\usepackage{soul}
\usepackage{authblk}

\newcommand{\po}{\left(}
\newcommand{\pf}{\right)}

\newcommand{\cco}{\llbracket}
\newcommand{\ccf}{\rrbracket}
\newcommand{\E}{\mathbb E}
\newcommand{\R}{\mathbb R} 
\newcommand{\T}{\mathbb T} 
\newcommand{\C}{\mathcal C}

\newcommand{\N}{\mathbb N} 

\newcommand{\M}{\mathcal M}

\newcommand{\dd}{\text{d}}
\newcommand{\na}{\nabla}

\newcommand{\pro}[1]{\mathbb{P}\left(#1~\right)}

\newcommand{\procz}[3]{\mathbb{P}_{#1}\left(#2~\middle|~#3\right)}

\newtheorem{thm}{Theorem}
\newtheorem{assu}{Assumption}
\newtheorem{lem}[thm]{Lemma}

\newtheorem{prop}[thm]{Proposition}

\newtheorem{rem}{Remark}

\title{Uniform convergence of the Fleming-Viot process in a hard killing metastable case}

\author[1]{Lucas Journel}
\author[1,2,3]{Pierre Monmarch\'e}
\affil[1]{LJLL, Sorbonne Université, Paris, France}
\affil[2]{LCT, Sorbonne Université, Paris, France}
\affil[3]{Institut universitaire de France (IUF)}

\begin{document}

\maketitle

\abstract{We study the long-time convergence of a Fleming-Viot process, in the case where the underlying process is a metastable diffusion killed when it reaches some  level set. Through a coupling argument, we establish the long-time convergence of the Fleming-Viot process toward some stationary measure at an exponential rate independent of $N$, the size of the system, as well as  uniform in time propagation of chaos estimates. }

\section{Introduction}

Given some open bounded domain $D\subset\R^d$, and some potential $U:\R^d\to \R_+$, we are interested in the process:
\begin{equation}\label{descente}
\dd X_t = -\na U(X_t)\dd t + \sqrt{2\varepsilon}\dd B_t
\end{equation}
with small  $\varepsilon>0$, killed when it reaches $\partial D$ the boundary of $D$. More precisely, write:
\begin{equation}\label{deathtime}
\tau_{\partial D} =\inf\left\{t\geqslant 0, X_t\notin D\right\}. 
\end{equation}
Denote by $\mathcal M^1(D)$ the set of probability measures on $D$, and $\mathbb P_\mu$ the law of the process~\eqref{descente}, with initial condition $\mu\in \mathcal M^1(D)$. Then we say that $\nu\in\mathcal M^1(D)$ is a quasi-stationary distribution (QSD) of the process \eqref{descente} if for all $t\geqslant 0$:
\[\procz{\nu}{X_t\in\cdot}{\tau_{\partial D}>t}=\nu.\]
It is shown in \cite{LelievreLeBris} that, under some mild assumptions on $U$ and $D$, the process \eqref{descente} admits a unique QSD, that we will denote by $\nu_{\infty}^{\varepsilon}$. It is also proven that there is convergence for all initial condition of the law of the process conditioned on its survival toward this QSD, namely, for all $\mu\in\mathcal M^1(D)$,
\[\procz{\mu}{X_t\in\cdot}{\tau_{\partial D}>t}\underset{t\rightarrow+\infty}\longrightarrow \nu_{\infty}^{\varepsilon}.\]
The fact that the process is killed when it exits a domain is classically referred to as a \emph{hard killing} case, by contrast with the \emph{soft killing} case where the process is killed  according to a  inhomogeneous Poisson process, as in \cite{journel20}.

The present work is concerned with the question of sampling the  QSD $\nu_{\infty}^{\varepsilon}$. More precisely, in practice, the QSD is approximated by the empirical measure of a system of interacting particles, called a Fleming-Viot (FV) process, at stationarity. This FV process is defined informally as follows: for a given $N$, let $X^1,\dots,X^N$ be $N$ independent diffusions until one of them reaches $\partial D$. The diffusion that has been killed then branches onto one of the $N-1$ remaining ones, chosen uniformly at random. In very general settings, it  is known  that if the initial condition consists in $N$ independent  random variables distributed according to a common law $\mu$, then for any time $t\geqslant 0$, we have:
\begin{equation}
\label{eq:propchaos}
a.s.\qquad
\pi^N(X^1_t,\dots,X^N_t)\overset{weak}{\underset{N\rightarrow\infty}\longrightarrow} \procz{\mu}{X_t\in\cdot}{\tau_{\partial D}>t},
\end{equation}
(see Section~\ref{RelWor} below) where
\begin{equation}\label{empiri}
\pi^N(x)= \frac{1}{N}\sum_{i=1}^N \delta_{x_i}
\end{equation}
stands for the empirical measure of a vector $(x_1,\dots,x_N)\in D^N$. This would simply be the law of large numbers if the particles were independent, which they are not due to the resurrection mechanism. For mean-field interacting particle systems as the FV process, such a convergence is known as a propagation of chaos phenomenon.

Two questions are addressed in this work. First, the long-time relaxation of the FV process toward its invariant measure: a quantitative convergence in the total variation distance sense at a rate independent from $N$ is stated in Theorem~\ref{thm}. Second, the propagation of chaos: Theorem~\ref{thm2} gives a quantitative version of \eqref{eq:propchaos}, with a bound uniform in time. Combining both results yields a quantitative estimate for the convergence  of the empirical measure of the FV process toward the QSD as $N,t\rightarrow \infty$.

As detailed below, these results are established under the condition that $D$ is a metastable state for the diffusion \eqref{descente}, in the sense that the mixing time of  \eqref{descente} within $D$ is shorter than the typical exit time from $D$. Mathematically speaking, this is reflected by the fact that $c^*$, the critical height within $D$, is smaller than $U_0$, the height of the boundary $\partial D$ (see below for the definition of $c^*$ and $U_0$), and the temperature $\varepsilon$ is small enough. This metastable context is typically the one where QSD are of interest, since in that case  the (non-conditional) law of the process is close to the QSD for times in intermediary scales between the mixing time and the extinction time. Moreover, it is exactly the context of some algorithms in molecular dynamics, such as the parallel replica algorithm presented in \cite{LelievreLeBris}, which involves the sampling of the QSD. In fact, for technical reasons, we will work under a stronger condition than simply $c^*<U_0$ (see condition~\eqref{strongdeathprob} in Assumption~\ref{assu3} below), which is also related to the metastability of $D$. While we haven't succeeded in this endeavor, we think  that the proof may possibly be modified to work only with the condition $c^*<U_0$, without the additional condition. As a particular case, let us notice that we don't need the additional condition when $d=1$ (see Lemma~\ref{unifexitevent}).

The paper is organized as follows. In the rest of this introduction, the main results are stated in Section~\ref{main} and discussed in view of previous related works in Section~\ref{RelWor}. Some preliminary properties of the FV process are studied in Section~\ref{sec:proof}, and the main theorems are proven in Section~\ref{conclu}. Finally, we prove in Section~\ref{technique} the technical lemma
which involves the additional condition.



\subsection{Main Result}\label{main}

Define the critical height $ c^* =c^*(U)$ of $U$ as $c^{\ast} = \sup_{x_1,x_2\in D}c(x_1,x_2)$ with 
\[c(x_1,x_2) = \inf\left\{\max_{0\leqslant t \leqslant1}U(\xi(t))-U(x_1)-U(x_2) \right\},\]
where the infimum runs over $\left\{ \xi\in\mathcal{C}\left(\left[0,1\right], D\right),\xi(0)=x_1, \xi(1)=x_2 \right\}$. The critical height $c^*$ represents the largest energy barrier the process has to cross in order to go from any local minimum  to any global one (within $D$).

The following conditions are enforced throughout all this work.

\begin{assu}\label{assu1}
\begin{itemize}
	\item $D\subset\R^d$ is open, bounded, connected and its boundary is $\mathcal C^2$.
	\item $U:\R^d \to \R_+$ is smooth on some neighborhood of $\overline{D}$.
	\item $\min_{D}U=0$ and   \begin{equation}\label{u0}
	    U_0 :=\min_{\partial D}U >c^* . 
	\end{equation}
	\item For $x\in\partial D$, denote by $n(x)$ the outward normal to $D$. For all $x\in \partial D$, 
	\begin{equation}\label{return}
	n(x) \cdot \na U(x) > 0.
	\end{equation}	 
\end{itemize}
\end{assu}

The condition $\min_{D}U=0$ is just a choice of normalisation since the process is unchanged if a constant is added to $U$. Under Assumption~\ref{assu1}, neglecting sub-exponential terms, for small $\varepsilon$, the mixing rate of the non-killed process~\eqref{descente} is known to be of order $e^{c^*/\varepsilon}$ (see \cite{HoStKu}) while,  according to the theory of Freidlin-Wentzell (see \cite{Freidlin-Wentzell}),  the exit time $\tau_{\partial D}$  is of order $e^{U_0/\varepsilon}$. As already mentioned, the condition $U_0>c^*$  thus describes a difference of timescales between the mixing time and the death time. More precisely, it is known that for any neighborhood $\mathfrak B_1$ of $\partial D$ and any $a<U_0$,
 $\sup_{x\in D\setminus \mathfrak B_1}\mathbb P_x(\tau_{\partial D}<e^{a/\varepsilon})$ vanishes with $\varepsilon$. In fact, we will need an even stronger uniformity in terms of the initial condition. For now, let us state it as an  assumption. We denote by $\varphi_t$ the flow associated to the deterministic gradient descent 
 \begin{equation}\label{def:flow}
 \partial_t \varphi_t(x) = -\na U \po \varphi_t(x)\pf,\quad \varphi_0(x) = x, 
 \end{equation}
 and $\varphi_t(\overline{D})$ the image of $\overline{D}$ by $\varphi_t$.  Notice that \eqref{return} implies that $\varphi_t(\overline{D}) \subset  D$ for $t>0$.
For a fixed Brownian motion $(B_t)_{t\geqslant 0}$, we denote $x\mapsto (X_t^x)_{t\geqslant 0}$ the stochastic flow associated to~\eqref{descente}, namely $(X_t^x)_{t\geqslant 0}$ solves the SDE with initial condition $x$ for all $x\in D$ and for all $T>0$, almost surely, $x\mapsto (X_t^x)_{t\in[0,T]}$ is continuous (for the topology of uniform convergence, see \cite[Theorem 37]{Protter} for the well-posedness of this flow).


\begin{assu}\label{assu2}
There exist $a>c^*$, $T_0>0$ and $\mathcal N$ a  neighborhood of $\varphi_{T_0}(\overline{D})$ such that, 
     denoting by    $\tau_{\partial D}(X^x)$ the first exit time from $D$ of  $(X_t^x)_{t\geqslant 0}$, we have 
    \begin{equation}\label{strongdeathprob_assum}
\mathbb P\left(\exists x\in \mathcal N, \tau_{\partial D}(X^x)<e^{a/\varepsilon}\right)\underset{\varepsilon\rightarrow 0}{ \longrightarrow} 0.
\end{equation}
\end{assu}

Notice that, if $(D_n)_{n\in\N}$ is an increasing sequence of sets whose union is $D$ and such that the distance between $D_n$ and $D^c$ is positive for all $n\in\N$, thanks to the continuity of $x\mapsto (X_t^x)_{t\in[0,T]}$ for any $T>0$, we may write
\[
\left\{ \exists x\in \mathcal N, \tau_{\partial D}(X^x)<e^{a/\varepsilon} \right\} = \cap_{n\in \N}\cup_{x\in \mathcal N\cap \mathbb Q^d} \left\{ \tau_{\partial D_n }(X^x) < e^{a/\varepsilon} \right\},
\]
so that the left-hand side is measurable (as a countable intersection of a countable union of event), and~\eqref{strongdeathprob_assum} makes sense.

We are able to prove that  \eqref{strongdeathprob_assum} is implied by the following condition, which in dimension $d>1$ strengthens~\eqref{u0}:

\begin{assu}\label{assu3}
One of the following is satisfied:
\begin{itemize}
    \item $d=1.$
    \item $c^*<(U_0-U_c)/2$, where \[ U_c = \lim_{t\rightarrow \infty} \sup_{x\in\partial D} U \po \varphi_t(x)\pf\,. \]
\end{itemize}
\end{assu}
Notice that $U_c$ is well-defined, as $t\mapsto \sup_{x\in\partial D} U \po \varphi_t(x)\pf $ is non-increasing and lower bounded by $0$.

As stated in Lemma~\ref{unifexitevent} below, Assumptions \ref{assu1} and \ref{assu3} together imply \ref{assu2}. 
However we don't think that Assumption~\ref{assu3} is sharp, and thus we state our main results in terms of Assumption~\ref{assu2}. 

We define the semi-group $(P_t)_{t\geqslant0}$ associated to a Markov process  $(X_t)_{t\geqslant 0}$ in $\R^d$ by:
\[P_tf(x)=\E_x(f(X_t))\]
for any bounded measurable function $f:\R^d\rightarrow \R$ and any $t\geqslant 0$, where $\E_x$ stands for the expectation under $\mathbb{P}_x=\mathbb{P}_{\delta_x}$. We denote $\cco 1,N\ccf = \left\{1,\cdots,N\right\}$.

Now, let us define rigorously the FV process, starting from some initial condition $\mu\in\mathcal M^1(D^N)$. Let $(I^i_n)_{1\leqslant i\leqslant N,n\in\mathbb N}$ be a family of independent random variables, where for $i\in\cco 1,N\ccf$, $I^i_n$ is uniform on $i\in\cco 1,N\ccf\setminus\left\{i\right\}$. Let $(B^i)_{i\in\cco 1,N\ccf}$ be $N$ independent Brownian motions, and $\textbf{X}_0=(X_0^1,\dots,X_0^N)$ be distributed according to $\mu$. Define $\bar X^i$ as the solution to:
\[\bar X^i_t = X_0^i -\int_0^t \na U(\bar X^i_s) \dd s + \sqrt{2\varepsilon}B^i_t\] 
and set
\[\tau_1=\min_i\inf\left\{t\geqslant 0, \bar X^i_t\notin D\right\}.\]
Then, denote by $i_1$  the index of the particle which exits the domain at time $\tau_1$. It is uniquely defined almost surely because, since the hitting time of the boundary has a density on $\R_+$, the probability that two particles hit the boundary at the same time is zero (this is true for the Brownian motion, and the general case follows from an application of the Girsanov theorem). For $i\neq i_1$, $0\leqslant t \leqslant \tau_1$, or $i=i_1$ and $1\leqslant t < \tau_1$, simply let:
\[X^i_t = \bar X^i_t \qquad \text{and}\qquad  X^{i_1}_{\tau_1}= \bar X^{I_1^i}_{\tau_1}.\]
This defines the process between times $0$ and $\tau_1$. The process is then defined on $\left(\tau_1,\infty\right)$ by induction: if the process is defined up to time $\tau_{n-1}$, we define it between time $\tau_{n-1}$ and $\tau_n$ in the same way, with $\textbf{X}_0$ replaced by $\textbf X_{\tau_{n-1}}$, $i_1$ by $i_n$ the index of the particle that hits $\partial D$ at time $\tau_n$, and $I_1^{i_0}$ by $I_n^{i_n}$. Thus, $(\tau_n)_n$ is the sequence of branching times of the process. 

Under Assumption~\ref{assu1}, the FV process $\textbf X = (X^1,\dots,X^N)$ is well-defined and does not explode in finite time, meaning that $\sup_n \tau_n =\infty$ almost surely, see \cite[Theorem 2.1]{villemonais2010interacting}. This defines a Markov process, and we denote by $P^{N,\varepsilon}=(P_t^{N,\varepsilon})_{t\geqslant 0}$ the associated semi-group.

A law $\mu\in\mathcal M^1(D^N)$ is said to be exchangeable if it is invariant by any permutation of the particles, i.e. $(X^{\sigma(i)})_{i\in\cco 1,N\ccf} \sim \mu$ if $(X^{i})_{i\in\cco 1,N\ccf} \sim \mu$ for all permutations $\sigma$ of $\cco 1,N\ccf$. For $k\in\cco 1,N\ccf$, we denote by $\mu^k\in\mathcal M^1(D^k)$ the marginal law of the $k$ first particles under $\mu$ (which, for exchangeable laws, is thus the marginal law of any subset of $k$ particles).

Our first main result concerns the long time behavior of the FV process. 

\begin{thm}\label{thm}
	Under Assumptions~\ref{assu1} and \ref{assu2}, there exist $\varepsilon_0 ,c,C>0$ such that, for all $\varepsilon\in(0,\varepsilon_0]$, $N\in\mathbb{N}$, $t\geqslant 0$, setting $t_{\varepsilon} = e^{a/\varepsilon}$, the following holds.
	\begin{enumerate}
	    \item For all $\mu,\nu\in\mathcal{M}^1(D^N)$, 
	\[\|\mu P^N_{t}-\nu P^N_{t}\|_{TV} \leqslant CN(1-c)^{t/t_\varepsilon}. \]
	\item The semi-group $P^{N,\varepsilon}$ admits a unique invariant measure $\nu^{N,\varepsilon}_\infty$, which is exchangeable. 
	\item For all exchangeable $\mu,\nu\in\mathcal M^1(D^N)$, for all $k\in\cco 1,N\ccf$, 
	\[\|(\mu P^{N}_{t})^k-(\nu P^{N}_{t})^k\|_{TV} \leqslant Ck(1-c)^{t/t_\varepsilon}. \]
	\end{enumerate}
\end{thm}

Our second result is a uniform in time propagation of chaos estimate.

\begin{thm}\label{thm2}
Under Assumptions~\ref{assu1} and  \ref{assu2}, there exists $\varepsilon_0>0$ such that for all compact set $K\subset D$, all $\varepsilon\in(0,\varepsilon_0]$, there exist $C_\varepsilon,\eta_\varepsilon>0$ such that for all   $\mu_0\in \mathcal{M}^1(D)$ satisfying $\mu_0(K)\geqslant 1/2$, all bounded $f:D\to\R_+$ and all $N\in\N$, 
\begin{equation*}
\sup_{t\geqslant0} \E\left( \left|\int_D f\dd \pi^N(\mathbf{X}_t) - \E_{\pi^N(\mathbf X_0)}\left(f\left(X_t\right)\middle|\tau_{\partial D} >t\right) \right| \right) \leqslant \frac{C_\varepsilon\|f\|_{\infty}}{N^{\eta_\varepsilon}},
\end{equation*}
where $X$ solves \eqref{descente},  $\tau_{\partial D}$ is defined in \eqref{deathtime}, $\pi^N$ in \eqref{empiri}, and $\textbf X$ is a FV-process with initial condition $\mu_0^{\otimes N}$.
\end{thm}

In Theorem~\ref{thm}, the dependency in $N$ of the speed of convergence is the same as for $N$ independent diffusion processes. Moreover, the dependency in $\varepsilon$ is also sharp. Indeed, in Assumption~\ref{assu2}, we can take $a$ arbitrarily close to $c^*$, which means that we get a mixing time smaller than $e^{(c^*+\delta)/\varepsilon}$ for any $\delta>0$, which is  the order of the mixing time for the non-killed process. However, as far as Theorem~\ref{thm2} is concerned, for independent processes, one would get from the Bienaym{\'e}-Chebyshev inequality   the explicit rate $1/\sqrt N$. This is indeed what is proven in \cite[Theorem 1]{Villemonais} for the FV process, but with a bound that depends on time. In other words, we improve the result of \cite{Villemonais} to a uniform in time bound, but at the cost of a loss in the rate in $N$. Notice that $\eta_{\varepsilon}$ may be made explicit by carefully following the proofs.

\subsection{Related works}\label{RelWor}

Fleming-Viot processes have first been introduce in the work of Fleming and Viot \cite{FlemingViot} and of Moran \cite{Moran}, in the study of population genetics models. Their use for the approximation of a QSD dates back to \cite{Burdzy1}, where the authors study the case of a Brownian motion in a rectangle. Since then, many results were proven in different cases, and for different questions (long-time convergence, propagation of chaos, existence of the FV process\dots).

In the case of a process in a countable state, the study began with \cite{ferrari2007}. The FV process is well defined here as soon as the death rate is bounded, and the authors showed under several conditions the uniqueness of the QSD, the convergence toward this QSD, the ergodicity of the FV process, and the propagation of chaos for finite time and at equilibrium. In \cite{groisman2012simulation}, the authors improved the propagation of chaos with a quantitative rate, introducing the $\pi$-return process. In \cite{CloezThai}, the rate of convergence of the FV process is proven to be independent of $N$ under strong assumptions, using coupling arguments similar to those of the present work or of \cite{journel20}. As soon as the set is finite, the existence of the FV process is immediate. In \cite{Asselah}, the propagation of chaos is proven for all times and for the stationary measure, with a stronger convergence. In \cite{LelievreReygner}, the convergence as $N\rightarrow\infty$ was refined with a central limit theorem.

For processes in a general space, many results are available. Finite time propagation of chaos is addressed in \cite{Burdzy2,GrigoKang,Villemonais,MicloDelMoral2}, with central limit theorems as $N\rightarrow\infty$ in \cite{Guyadersoft,Guyaderhard}. Then, uniform in time propagation of chaos and long-time convergence are established in \cite{MicloDelMoral,DelMoralGuionnet,Rousset}. The long-time convergence is established when the underlying process is a Brownian motion in \cite{Burdzy1}. If the killing-rate is smooth and bounded, then the well-definiteness of the FV process is obvious, but the non-explosion in the hard killing case has been studied in \cite{bieniek2009nonextinction,bieniek2011extinction}, and along with long time convergence in \cite{villemonais2010interacting}.

Other methods of approximation of a QSD have been developed in discrete and continuous cases in \cite{benaim2015,BenaimCloezPanloup,BCV}, based on self-interacting processes. Study of the conditioned process and its long-time limit has also been studied for a decade by Champagnat, Villemonais and coauthors, in \cite{ChampVilK2018,CV2020,Bansaye,DelMoralVillemonais2018}.

The coupling method used in the present work has been applied in several works about interacting particles systems (not only FV processes) such as \cite{Monmarche,CloezThai,journel20}. These works are based on a perturbation approach where the interaction is assumed to be small enough with respect to the mixing properties of the underlying Markov process. In particular, in our framework, we do not know if the conclusion of Theorem~\ref{thm} is true for any $\varepsilon>0$ and not only in the low temperature regime.

Our work follows the similar study \cite{journel20} in the soft killing case. In fact, as mentioned in the latter, the main motivation of this first work was to set up in a simpler case a method that would then be used to tackle the hard killing case, which was from the beginning the main objective. Besides, \cite{journel20} also addresses the question of the time-discretization of \eqref{descente}, but in the present work we focus on the convergence in $N$ and $t$ and thus we only consider the continuous time dynamics for clarity.

\section{Preliminary results}
\label{sec:proof}

For any two probability measures $\mu, \nu$, we call $(X,Y)$ a coupling of $\mu$ and $\nu$ if the law of $X$ (resp. $Y$) is $\mu$ (resp. $\nu$). For any distance $\mathbf d$ on a set $E$ (here $E=D$ or $E=D^N$), the associated Kantorovich distance on $\mathcal M^1(E)$, is defined by
\[W_{\mathbf d}(\mu,\nu) = \inf\left\{\mathbb E\po \mathbf d(X,Y)\pf, (X,Y) \text{ coupling of }\mu \text{ and }\nu \right\}.\]
We say that $(X,Y)$ is an optimal coupling if $W_{\mathbf d}(\mu,\nu) = \mathbb E\po \mathbf d(X,Y)\pf$. The existence of an optimal coupling results from \cite[Theorem 4.1]{villani}. Given a Markov semi-group $P$, we call a coupling of $(\mu P_t)_{t\geqslant 0}$ and $(\nu P_t)_{t\geqslant 0}$ a stochastic process $(X_t,Y_t)_{t\geqslant 0}$ such that $(X_t)_{t\geqslant 0}$ and $(Y_t)_{t\geqslant 0}$ are Markov processes of semi-group $(P_t)_{t\geqslant 0}$ and initial condition $\mu$ for $X$ and $\nu$ for $Y$. In particular, we have that for such a coupling and all $t\geqslant 0$, \[W_{\mathbf d}(\mu P_t,\nu P_t) \leqslant \E\left( \mathbf d(X_t,Y_t) \right).\]We also say that the processes $X$ and $Y$ have coupled at time $t\geqslant 0$ if $X_t=Y_t$. Finally, in the case where $\mathbf d(x,y)=2\mathbbm{1}_{x\neq y}$, we recover the total variation distance which we write $W_{\mathbf d}(\mu,\nu)=\|\mu-\nu \|_{TV}$.

The proof of our theorems relies on the construction of a coupling of $(\mu P_t^{N,\varepsilon})_{t\geqslant 0}$ and $(\nu P_t^{N,\varepsilon})_{t\geqslant 0}$. This coupling will yield that $P_t^{N,\varepsilon}$ is a contraction for some particular distance defined in Section~\ref{conclu}.

In order to do this, we first need some preliminary results, which is the subject of this section. We start by studying the mixing properties of the non-killed process in Subsection~\ref{coupnonkil} by embedding $D$ into a torus. In Subsection~\ref{Lyapunov}, We construct a Lyapunov functional for each particle. In Subsection~\ref{partnearbound}, using the Lyapunov functional, we study the number of particle that may stay near the boundary of the domain.

In the rest of the paper, we fix some $a\in\left(c^*,U_0\right)$ satisfying assumption~\ref{assu2}  and set $t_\varepsilon=e^{a/\varepsilon}$. Furthermore, bold letters will always denote particle systems, in the sense that $\textbf{X}$ can always be written $\textbf{X}=(X^1,\dots,X^N)$ where for all $1\leqslant i\leqslant N$, $X^i\in D$.

\subsection{Coupling of the non-killed diffusion}\label{coupnonkil}

In this section, we show that we are able to couple two diffusions solution of \eqref{descente} on a torus in total variation distance in a time $t_\varepsilon$ with a probability that goes to 1 as $\varepsilon$ goes to 0, uniformly on $D$. Since we are studying a process killed at the boundary of $D$, we are not interested in what the potential might look like outside of $D$. Consider some torus $\T^d=\left(\R/2L\mathbb Z\right)^d$, with $L$ big enough so that as a subset of $\R^d$ (meaning seeing $\T^d$ as $\left[-L,L\right[^d$), we have that $D\subset \T^d$. Then consider some periodic potential $\tilde U:\R^d\rightarrow\R_+$, equal to $U$ on $D$ as a periodic function, and such that $c^*(\tilde U)=c^*(U)$, where $c^*(\tilde U)$ is defined as $c^*(U)$ with $U$ replaced by $\tilde U$. Such a function exists, as shown in \cite[Section 4]{FourTar}. We still denote by $\tilde U$ the associated function on $\T^d$, and this potential defines a diffusion on $\T^d$ as:

\begin{equation}\label{descente2}
\dd \tilde X_t = -\na \tilde U(\tilde X_t)\dd t + \sqrt{2\varepsilon}\dd B_t.
\end{equation}

We note $\tilde P$ its semi-group. If we see $\tilde X$ as a process in $\R^d$, then we have that $X_t=\tilde X_t$ for all $t\leqslant \tau_{\partial D}$, where $\tau_{\partial D}$ is the death time \eqref{deathtime}.

Now, we construct a coupling for the process $\tilde X$, for all initial condition $(x,y)\in (\T^d)^2$. To do this, we use Sobolev and Poincar\'e inequalities. The Sobolev inequality is used for ultra-contractivity, whereas the Poincar\'e inequality is used to get an optimal convergence rate for the process \eqref{descente2}. Let $\mu_\varepsilon$ denote the probability measure on $\T^d$:\[\mu_\varepsilon(\dd x) = \mathcal Z^{-1} e^{-\tilde U(x)/\varepsilon}\dd x\] where $\mathcal Z$ is the normalization constant. Recall those inequalities:

\begin{lem}\label{Sob}
	$\mu_\varepsilon$ satisfies a Poincar\'e and a Sobolev inequality: there exist $p>2$, $C,\lambda_\varepsilon>0$,  such that
	\[\varepsilon\ln\left(\lambda_\varepsilon\right)\rightarrow -c^*\]	
	 as $\varepsilon\rightarrow 0$, and for all smooth $f:\T^d\mapsto\R$ with $\int_{\T^d} f \dd \mu_\varepsilon = 0$
	\[\lambda_{\varepsilon}\int_{\T^d} f^2 \dd \mu_\varepsilon \leqslant\int_{\T^d} |\nabla f|^2 \dd \mu_\varepsilon\ \text{ (PI)},\]
	\[\left(\int_{\T^d} f^p \dd \mu_\varepsilon\right)^{\frac{2}{p}} \leqslant Ce^{\|\tilde U\|_{\infty}/\varepsilon}\left(\int_{\T^d} f^2 \dd \mu_\varepsilon + \int_{\T^d} |\nabla f|^2 \dd \mu_\varepsilon\right)\ \text{ (SI)}.\]
	 Moreover, for all $t>0$, the law of $\tilde X_t$ with initial condition $x$ has a density $h^\varepsilon_t(x,\cdot)$ with respect to $\mu_\varepsilon$, and both inequalities together imply the existence of some constant $\tilde C>0$ such that, for all $t\geqslant 1$ and $\varepsilon>0$,
	\begin{equation}\label{TVCONV}
	\|h^\varepsilon_t(\cdot,\cdot)-1\|_{\infty} \leqslant \tilde C e^{\tilde C \varepsilon^{-1} -\lambda_{\varepsilon}t}.
	\end{equation}
\end{lem}

\begin{proof}
	The  Poincar{\'e} inequality, as well as the asymptotic on $\lambda_\varepsilon$, have been proven in \cite{HoStKu}.
	The uniform measure on $\T^d$ satisfies a Sobolev inequality, see \cite[section 6]{BakryGentilLedoux}. Then we can write:
	\begin{align*}
	\left(\int_{\T^d} f^p \dd \mu_\varepsilon\right)^{\frac{2}{p}} &\leqslant \mathcal{Z}^{-\frac{2}{p}} \left(\int_{\T^d} f^p\right)^{\frac{2}{p}} \\ & \leqslant C\mathcal{Z}^{-\frac{2}{p}}  \left( \int_{\T^d} f^2 + \int_{\T^d} |\na f|^2 \right) \\ &\leqslant C\mathcal{Z}^{1-\frac{2}{p}} e^{\|\tilde U\|_{\infty}/\varepsilon} \left( \int_{\T^d}f^2\dd \mu_\varepsilon + \int_{\T^d} |\na f|^2 \dd\mu_\varepsilon \right).
	\end{align*}
	Since $\mathcal{Z}$ is bounded above by the volume of $\T^d$,  $\mathcal{Z}^{1-\frac{2}{p}}$ is bounded uniformly over $\varepsilon$ because $p>2$ (and thus $1-\frac{2}{p}>0$). Therefore, we have the Sobolev inequality with the said constant.
	The last two points are \cite[Theorem 6.3.1 and Proposition 6.3.4]{BakryGentilLedoux}.

\end{proof}

\begin{lem}\label{coupopt}
Let $a>c^*$.
Under Assumption~\ref{assu1}, there exists $\varepsilon_0>0$, such that for all $0<\varepsilon<\varepsilon_0$, there exists $c_{\varepsilon}>0$ such that for all $x,y\in D$, there exists a coupling $(\tilde X_t,\tilde Y_t)_{t\geqslant 0}$ of $(\delta_x\tilde P_t)_{t\geqslant 0}$ and $(\delta_y\tilde P_t)_{t\geqslant 0}$ such that, writing $t_{\varepsilon} = e^{a/\varepsilon}$: \[\pro{\tilde X_{ t_\varepsilon}= \tilde Y_{ t_\varepsilon}}\geqslant c_\varepsilon.\]Moreover, as $\varepsilon\rightarrow0$, \[c_\varepsilon\rightarrow 1.\]
\end{lem}

\begin{proof}
	We start by bounding the total variation distance between the law of $\tilde X_t$ and the equilibrium $\mu_\varepsilon$. Recall that $h^\varepsilon_t(x,\cdot)$ denotes the density of the law of $\tilde X_t$  with respect to $\mu_\varepsilon$. For $x\in D$, using \eqref{TVCONV}, we have:
	\[\|\delta_x\tilde P_t - \mu_\varepsilon\|_{TV} = \int_{R^d}|h^\varepsilon_t-1| \dd \mu_\varepsilon \leqslant  \|h^\varepsilon_t(\cdot,\cdot)-1\|_{\infty} \leqslant \tilde C e^{\tilde C \varepsilon^{-1} -\lambda_{\varepsilon}t}.\]
	Since $t_\varepsilon=e^{a/\varepsilon}$, and $\varepsilon\ln(\lambda_\varepsilon)\rightarrow -c^*$ as $\varepsilon\rightarrow0$, we have that $\lambda_\varepsilon t_\varepsilon\geqslant e^{b/\varepsilon}$, for $\varepsilon$ small enough and all $b<a-c^*$. Then, we may fix some $b<a-c^*$ and let:
	\[c_\varepsilon = 1-2\tilde C e^{\tilde C \varepsilon^{-1} -e^{b/\varepsilon}},\]so that $\lim_{\varepsilon\rightarrow0}c_\varepsilon= 1$ and for all $x,y\in D$, \[\|\delta_x\tilde P_t - \delta_y \tilde P_t\|_{TV}\leqslant \|\delta_x\tilde P_t - \mu_\varepsilon\|_{TV}+\|\delta_y\tilde P_t - \mu_\varepsilon\|_{TV} \leqslant 1-c_\varepsilon.\]
	The existence of the coupling of the trajectories results from the total variation distance properties, see for example \cite[Lemma 9]{Monmarche}.
\end{proof}

\subsection{Lyapunov functional}\label{Lyapunov}

In order to show the convergence of the particle system, we need first the construction of some Lyapunov function for each particle, that is to say some function of one particle that decreases in average over time, as long as it starts large enough. This is the goal of this section. We first need to strengthens Assumption~\ref{assu2}, allowing for a larger number of initial conditions than in \eqref{strongdeathprob_assum}.

\begin{lem}
Under Assumption~\ref{assu1} and~\ref{assu2}, for all $\mathfrak B_1\subset D$ such that $\mathfrak B_1\cup(\R^d\setminus D)$ is a neighborhood of $\partial D$,     \begin{equation}\label{strongdeathprob}
\bar p_\varepsilon:=\mathbb P\left(\exists x\in D\setminus \mathfrak B_1, \tau_{\partial D}(X^x)<e^{a/\varepsilon}\right)\underset{\varepsilon\rightarrow 0}{ \longrightarrow} 0.
\end{equation}
\end{lem}

\begin{proof}
Recall the definition of $T_0$ and $\mathcal N$ from Assumption~\ref{assu2}, fix some $\mathfrak B_1\subset D$ such that $\mathfrak B_1\cup(\R^d\setminus  D)$ is a neighborhood of $\partial D$ , and denote

\[
\delta = \min\left( \inf_{t\in[0,T_0]} \mathrm{dist}( \varphi_t(D\setminus \mathfrak B_1), D^c),  \mathrm{dist}(\varphi_{T_0}(\overline{D}),\mathcal N^c)  \right),
\]
where $\mathrm{dist}$ denotes the distance between two sets, which is positive (by continuity, compactness and \eqref{return}). We have:
\[
\varphi_t(x) - X_t^x = -\int_0^t \na U(\varphi_s(x)) - \na U(X_s^x) \dd s - \sqrt{2 \varepsilon} B_t
\]
so that if $L$ denotes some Lipschitz constant of $\na U|_D$, we get for all $T>0$:
\[
\sup_{0\leqslant t \leqslant T}|\varphi_t(x) - X_t^x| \leqslant L \int_0^T \sup_{0\leqslant t \leqslant s}|\varphi_t(x) - X_t^x| \dd s + \sqrt{2\varepsilon} \sup_{0\leqslant t \leqslant T}|B_t|,
\]
and Grönwall's lemma yields
\[
\sup_{0\leqslant t \leqslant T_0}|\varphi_t(x) - X_t^x|\leqslant \sqrt{2\varepsilon}e^{LT_0}\sup_{0\leqslant t \leqslant T_0} |B_t|.
\]
In particular, we have that 
\[
\lim_{\varepsilon\rightarrow 0}\mathbb P\left( \mathcal A \right) = 1, \qquad \mathcal A = \left\{ \sup_{0\leqslant t \leqslant T_0}\sup_{x\in D}|\varphi_t(x) - X_t^x|\leqslant \delta/2 \right\}.
\]
On the event $\mathcal A$, $\tau_{\partial D}(X^x) > T_0$ and  $X_{T_0}^x\in \mathcal N$ for all $x\in D\setminus \mathfrak B_1$, and using the Markov property at time $T_0$, we have that 
\[
\mathbb P\left(\exists x\in D\setminus \mathfrak B_1, \tau_{\partial D}(X^x)<e^{a/\varepsilon}\right) \leqslant  \mathbb P\po\mathcal A^c\pf + \mathbb P\po \exists x\in \mathcal N, \tau_{\partial D}(X^x)<e^{a/\varepsilon} \pf \underset{\varepsilon\rightarrow 0}{ \longrightarrow} 0, 
\]
which concludes the proof.
\end{proof}

For all $q>0$, write
\[
F_q=\{y\in D, \mathrm{dist}(y,\partial D)> q\}.
\]
From \cite[Lemma 14.16]{gilbarg1977elliptic}, there exists $q_0>0$ such that  $x \mapsto d_{\partial D}(x):=  \mathrm{dist}(x,\partial D)$ the Euclidean distance to $\partial D$ is $\mathcal C^2$ over $D\setminus F_{q_0}$. Moreover, thanks to \eqref{return} in Assumption~\ref{assu1}, we may take $q$ small enough so that  
\begin{equation}\label{eq:zeta}
\kappa:= \inf_{x\in D\setminus F_{q_0}} \na d_{\partial D}(x)\cdot \na U(x)  >0.
\end{equation}
For the remainder of this work, we fix some $\mathfrak B_1\subset D$ such that $\mathfrak B_1\cup(\R^d\setminus D)$ is a neighborhood of $\partial D$ and 
\begin{equation}
    \label{eq:condB1}
    \mathfrak B_1 \subset D\setminus F_{q_0/4},\qquad \min_{\mathfrak B_1} U>a.
\end{equation}
This second condition will be used in Section~\ref{conclu}, and is possible because $a<U_0$ (recall we fixed some $a\in\left(c^*,U_0\right)$ satisfying assumption~\ref{assu2}). In the rest of this work, $\bar{p}_\varepsilon$ is given by~\eqref{strongdeathprob} with this fixed $\mathfrak B_1$.

The construction of the Lyapunov function also relies on a result first shown in \cite{GrigoKang}, similar to Ito's formula. The process~\eqref{descente} has a generator $\mathcal L$ defined for all smooth function $f:\R^d\to\R$ with compact support as:
\begin{equation}
\mathcal L f= \varepsilon \Delta f - \na U\cdot \na f.
\end{equation} 
For a smooth function $f:D^N\to\R_+$, and $1\leqslant i\leqslant N$, write $\mathcal L_{x_i}$ for the generator $\mathcal L$, acting only on the i-th variable:

\begin{equation*}
    \mathcal L_{x_i}f = \sum_{j=1}^d \varepsilon \partial^2_{(x_i)_j}f - \partial_j U(x_i) \partial_{(x_i)_j}f.
\end{equation*}
Write as well $(\tau_n^i)_n$ for the sequence of death times of particle $i$:
\begin{equation}\label{eqf:deftauni}
    \tau^i_0=0,\quad \tau^i_{n+1} = \inf\left\{t> \tau^i_n, X_{t-}^i\in\partial D\right\},
\end{equation}
and $R_i$ for the point process corresponding to the jumps of this particle:
\[R_i(t)=\sum_{n=1}^{\infty}\mathbbm{1}_{\tau^i_n\leqslant t}.\]
For all $\textbf x\in D^N$, $1\leqslant i\neq j\leqslant N$, write:
\[ x^{ij}_k = \left\{ \begin{matrix} x_k\text{ if }k\neq i \\ x_j \text{ else }\end{matrix} \right..\]

\begin{prop}[\cite{GrigoKang}, Proposition 1]\label{Ito}
	Let $N\in\mathbb N$,  $f\in\C^0\left(\bar{D}^N\right)\bigcap\C^{\infty}\left(D^N\right)$. Denote by \[\mathbf{R}(f)(t)=\sum_{i=1}^N\frac{1}{N-1}\sum_{j\neq i}\int_0^t \left(f(\mathbf{X}^{ij}_{s-})-f(\mathbf{X}_{s-})\right)\dd R_i(s)\] the finite variation process of the jump part, and \[Q(f)(t)=\int_0^t \sum_{i=1}^N \mathcal L_{x_i}f(\mathbf{X}_s)\dd s\] the finite variation process of the diffusion part. Then there exists a martingale $\M(f)$ such that almost surely for all $t\geqslant 0$: \[f(\mathbf{X}_t)-f(\mathbf{X}_0)= \mathbf R(f)(t) + Q(f)(t) + \M(f)_t. \]
	Moreover, 
	\[\mathbf R(f)(t)+\M(f)_t = \sum_{i=1}^N\int_0^t \na_{x_i}f(\textbf X_s)\cdot \dd B^i_s + \sum_{n,\tau_n\leqslant t}f(\textbf X_{\tau_n})-f(\textbf X_{\tau_n-}),\]
	where the $(B^i)$'s are the Brownian motion used in the definition of the FV process, and the $(\tau_n)$ are the death times.
\end{prop}

This decomposition will allow us to prove the existence of the Lyapunov functional for each particle. 

\begin{lem}\label{Lyapu}
	Under Assumption~\ref{assu1} and~\ref{assu2}, for all $V_0 > 9$, there exist $\varepsilon_0>0$ and a smooth function $V:\bar D\mapsto[1,+\infty)$ such that $V$ is constant equal to $V_0$ on $\partial D$, $\sup_D V=V_0$, $\mathfrak B_1\subset \{V >3V_0/4\}$ and for all $N\in\N$, $0<\varepsilon<\varepsilon_0$, $\textbf x\in D^N$, and $1\leqslant i\leqslant N$, we have:
	\begin{equation}\label{lyapu}
	\E_{\textbf x}\left(V\left(X^i_{t_\varepsilon}\right)\right) \leqslant \gamma_{\varepsilon}V(x_i) + C_1(1-\gamma_{\varepsilon})
	\end{equation}
 where  $\gamma_\varepsilon \in(0,1)$ is independent from $\textbf{x}$ and vanishes as  $\varepsilon\rightarrow 0$, and $C_1=9/4$.
\end{lem}

The value $V_0$ of $V$ on the boundary is a fixed parameter that will be chosen in Section~\ref{conclu}. In the rest of the work we will use the notation $C_1$ instead of its explicit value since the latter  follows from an arbitrary choice in our proof (as is the constraint that $V_0>9$) and, for instance,  it will be clearer than with the explicit value when we consider $2C_1$ that this is related with Lemma~\ref{Lyapu}.

\begin{proof}
Fix some $V_0 >9$, and let $f:\left[0,\|d_{\partial D}\|_{\infty}\right]\mapsto \R_+$ be some smooth non-increasing function satisfying the following conditions:
\begin{itemize}
	\item for all $u\geqslant q_0$, $f(u)=1$,
	\item $f(0)=V_0$,
    \item $f(q_0/2) \leqslant 2$
	\item $\max_{0 \leqslant u\leqslant q_0/2}f'(u)<0$,
    \item $f(q_0/4) > 3/4 V_0$
\end{itemize}
Then, set $V(x) = f(d_{\partial D}(x))$ for all $x\in D$. The function $V$ is smooth, constant equal to $V_0=f(0)$ on $\partial D$, bounded above by $V_0$, and the fact that $\mathfrak B_1 \subset \{V > 3V_0/4\}$ follows from the last condition on $f$ and~\eqref{eq:condB1}. We have for all $x\in D\setminus F_{q_0}$ (since $d_{\partial D}$ is $\mathcal C^2$ on this set):
 \begin{equation}\label{eq:LV}
     \mathcal{L}V = f'(d_{\partial D})\left(\varepsilon\Delta d_{\partial D}-\na d_{\partial D}\cdot\na U\right) + \varepsilon f''(d_{\partial D})|\na d_{\partial D}|^2.
 \end{equation}
Using \eqref{eq:zeta}, we may consider $\varepsilon_0$ small enough so that on $D\setminus F_{q_0/2}$ we have, for all $\varepsilon<\varepsilon_0$: 
\[
\varepsilon f'(d_{\partial D})|\Delta d_{\partial D}|+\varepsilon |f''(d_{\partial D})||\na d_{\partial D}|^2 \leqslant -1/2 f'(d_{\partial D})\na d_{\partial D}\cdot \na U.
\] 
We then have for all $x\in D\setminus F_{q_0/2}$:
\[
\mathcal LV(x)\leqslant -\omega V(x)
\]
where
\[
\omega = -\frac{\kappa}{2V_0}\max_{0 \leqslant u\leqslant q_0/2}f'(u)
\]
is independent of $\varepsilon$. Second, for $x\in F_{q_0/2} \setminus F_{q_0}$, it still holds 
\[
f'(d_{\partial D}(x)) \nabla d_{\partial D}(x) \cdot \na U(x) \geqslant 0,
\]
so that, writing 
\[
C=\sup\{f'(d_{\partial D})\left( \Delta d_{\partial D}\right) +   f''(d_{\partial D})|\na d_{\partial D}|^2 ,\  x \in D\setminus F_{q_0}\}
\]
 we get from \eqref{eq:LV}, for all $x\in F_{q_0/2}\setminus F_{q_0}$
\[
\mathcal L V(x)  \leqslant \varepsilon C \leqslant  \varepsilon C + 2\omega  - \omega V(x).\]
This inequality is thus true for all $x\in D$, as $V$ is constant on $F_{q_0}$. 
Hence, we may plug it into the formula of Proposition~\ref{Ito} with $f(\textbf x)=V(x_i)$.  Recall the definition of $(\tau^i_n)_{n\in\N}$ from~\eqref{eqf:deftauni}. For $n\in\N$, using  the fact that $V$ is maximal on the boundary of $D$, $\textbf R(V) \leqslant 0$, and hence for all $0\leqslant s\leqslant t$:
\[\E\left(V\left(X^i_{t\wedge \tau_n^i}\right)\right) - \E\left(V\left(X^i_{s\wedge \tau_n^i}\right)\right) \leqslant \int_s^t \left( -\omega \E\left(V\left(X^i_{u\wedge \tau_n^i}\right)\right) + 2\omega + \varepsilon C \right)\dd u.\]
Writing $g(t)=\E\left(V\left(X^i_{t\wedge \tau_n^i}\right)\right) - 2 - \varepsilon C/\omega$, we have for $t\geqslant 0$ and $h>0$:
\[g(t+h)-g(t)\leqslant -\omega\int_t^{t+h} g(u)\dd u. \]
Assume for now that $g$ is continuous. Fix $\eta>0$ and write:
\[T_\eta=\min\left\{t\geqslant 0, g(t)\geqslant g(0)e^{\omega t}+\eta\right\}.\]
Suppose that $T_\eta<\infty$ and write:
	\[s_\eta=\max\left\{ 0\leqslant s\leqslant T_\eta, g(s)\leqslant g(0)e^{\omega s}\right\}.\]Then by continuity we have:
\begin{multline*}g(0)e^{-\omega T_\eta } + \eta - g(0)e^{-\omega s_\eta} = g(T_\eta) - g(s_\eta) \leqslant -\omega\int_{s_\eta}^{T_\eta}g(u)\dd u \\< -\omega\int_{s_\eta}^{T_\eta}g(0)e^{-\omega u}\dd u = g(0)e^{-\omega T_\eta } - g(0)e^{-\omega s_\eta},
\end{multline*}
hence necessarily $T_\eta=\infty$, for all $\eta>0$, and thus for all $t\geqslant 0$, \[g(t)\leqslant g(0)e^{-\omega t},\]and
\[
\E\left(V\left(X^i_{t_\varepsilon\wedge \tau_n^i}\right)\right) \leqslant \gamma_{\varepsilon}V(x_i) + \frac{2\omega + \varepsilon C}{\omega}\left(1-\gamma_{\varepsilon}\right),
\]
with $\gamma_{\varepsilon}=e^{-\omega t_\varepsilon}$. Because $V$ is maximal at the boundary, $t\mapsto V\po X^i_t \pf $ is lower-semicontinuous. Hence, since $V\geqslant 0$ and $\tau_n^i \rightarrow \infty$ as $n\rightarrow \infty$, using Fatou's lemma,
\begin{multline*}
\E\left(V\left(X^i_{t_\varepsilon}\right)\right) \leqslant \E \left(\liminf_{n\rightarrow\infty} V\left(X^i_{t_\varepsilon\wedge \tau_n^i}\right)\right) \\\leqslant \liminf_{n\rightarrow\infty} \E \left(V\left(X^i_{t_\varepsilon\wedge \tau_n^i}\right)\right) \leqslant \gamma_{\varepsilon}V(x_i) + \frac{2\omega + \varepsilon C}{\omega}\left(1-\gamma_{\varepsilon}\right).
\end{multline*}
Since $V_0>9$, we may take $\varepsilon_0$ small enough so that, for all $\varepsilon < \varepsilon_0$: \[4\frac{2\omega + \varepsilon C}{\omega} < 9 < V_0,\]
and hence the result with $C_1=9/4$.
 
We are left to show that, for all $n\in\N$, $t\mapsto \E(V(X^i_{t\wedge \tau^i_n}))$ is continuous. Write $V^i(\textbf x)=V(x_i)$. Then from Proposition~\ref{Ito}, we get that:
\[\E(V(X^i_{t\wedge \tau^i_n}))=\E(V(X^i_0)) + \E(Q(V^i)(t\wedge \tau^i_n)) + \E(\textbf R(V^i)(t\wedge \tau^i_n)).\]
Since $\mathcal L_{x_i}V$ is bounded,  $t\mapsto \E(Q(V^i)(t\wedge \tau^i_n))$ is continuous. For $0\leqslant s \leqslant t$, we may write
\[
\E(\textbf R(V^i)(t\wedge \tau^i_n)) - \E(\textbf R(V^i)(s\wedge \tau^i_n)) = \frac{1}{N-1}\sum_{k=1}^n\sum_{j\neq i} \E \po \po V(X^j_{\tau_k^i}) - V(X^i_{\tau_k^i-}) \pf \mathbbm 1_{s \leqslant \tau^i_k \leqslant t} \pf ,
\]
and hence
\[
\left|\E(\textbf R(V^i)(t\wedge \tau^i_n)) - \E(\textbf R(V^i)(s\wedge \tau^i_n))\right| \leqslant 2n\|V\|_{\infty}\mathbb P(X^i\text{ dies between time }s\text{ and }t ).
\]
The law of the death times have a density with respect to the Lebesgue measure. Hence we have that
\[
\lim_{s\rightarrow t}\mathbb P(X^i\text{ dies between time }s\text{ and }t ) = 0,
\]
and this implies the continuity of $t\mapsto \E(\textbf R(V^i)(t\wedge \tau^i_n))$, which concludes the proof.
\end{proof}

\begin{rem}\label{remarque}
In particular, since $V_0>9$, $\mathfrak B_1\subset \left\{V>m\right\}$
with $m = 3V_0/4 >3C_1$.
\end{rem}

We are now interested in the death probability of a particle.

\begin{prop}\label{prop:couplage}
	Under Assumption~\ref{assu1}, denote by $(X_t)$ the diffusion \eqref{descente}, and 
	\[\tau_{\partial D}=\inf\left\{ t\geqslant 0; X_t \notin D\right\}.\]
	Consider any $C_2\in\left(2C_1,4C_1\right)$, where $C_1$ is given in Lemma~\ref{Lyapu}. Then
	we have:
	\begin{equation}\label{deathprob}
	p_{\varepsilon} :=\sup_{x\in\left\{V\leqslant C_2\right\}}\mathbb{P}_x(\tau_{\partial D} <t_\varepsilon) \rightarrow0,
	\end{equation}
	as $\varepsilon\rightarrow0$.
\end{prop}

\begin{proof}
	Since $t_\varepsilon = e^{a/\varepsilon}$ with $a<U_0$, this is the theory of Freidlin-Wentzell, see \cite[Chapter 6, Theorem 6.2]{Freidlin-Wentzell}. Freidlin and Wentzell didn't state the uniformity, but it follows from their proof.
\end{proof}

\subsection{Particles near the boundary}\label{partnearbound}

We want to control the number of particles which are close to the boundary of $D$ after a time $t_\varepsilon$. Consider the neighborhood 
\begin{equation}
    \label{eq:B}
\mathfrak B=\left\{V>3C_1\right\}
\end{equation}
of $\partial D$, where $C_1$ is the constant from Lemma~\ref{Lyapu}.
For $\textbf x=(x_1,\cdots,x_N)$, write:
\begin{equation}\label{numbbound}
A(\mathbf x)= \#\left\{i\in \llbracket1,N\rrbracket; x_i\in \mathfrak B\right\},
\end{equation}
where $\#$ stands for the cardinality of a set. We show that at time $t_{\varepsilon}$, the number of particles close to the boundary, $A(\mathbf{x})$, is a small fraction of $N$ with high probability as $N$ goes to infinity or as $\varepsilon$ goes to $0$. 

\begin{lem}\label{boundary}
	For all $\alpha>0$, there exists $\varepsilon_0>0$ such that for all $\varepsilon<\varepsilon_0$, there exists $q_\varepsilon>0$ such that for all $N\in \mathbb N$ and $\mathbf x\in D^N$:
	\begin{equation}\label{probbound}
	\mathbb{P}_{\mathbf{x}} (A(\mathbf X_{t_\varepsilon})>\alpha N) \leqslant q_\varepsilon^N
	\end{equation}
	and $q_\varepsilon \rightarrow0$ as $\varepsilon\rightarrow0$.
\end{lem}

\begin{proof}
The idea for the proof is the following: we want to compare the evolution of $V(X^i_t)$ and Ornstein-Uhlenbeck processes with small variance. If we had $N$ independent diffusions, the result would derive from a simple enumeration. But then the interaction through jumps can only make the Lyapunov decrease.
	From  Proposition~\ref{Ito} and the proof of Lemma~\ref{Lyapu}, we have that almost surely for all $1\leqslant i \leqslant N$ and $t\geqslant 0$:
	\[V\left(X^i_t\right) \leqslant V(x_i) + \int_0^t \left(-\omega V\left(X^i_s\right)+ \omega C_1 \right) \dd s+ \sqrt{2\varepsilon}\int_0^t \nabla V\left(X^i_s\right)\cdot \dd B_s^i.\]
	for some $\omega>0$ (independent from $\varepsilon$). Now introduce for $1\leqslant i \leqslant N$ the process:
	\[K_t^i = V(x_i)e^{- \omega t} + C_1(1-e^{-\omega t}) + \sqrt{2\varepsilon}\int_0^t e^{ \omega (s-t)} \na V(X^i_s)\cdot\dd B_s^i,\]
	which solves
	\[K_t^i =  V(x_i) + \int_0^t \left(-\omega K_s^i + \omega C_1 \right) \dd s+ \sqrt{2\varepsilon}\int_0^t \nabla V\left(X^i_s\right)\cdot \dd B_s^i.\]

	From proposition~\ref{Ito}, we have that:
	\[V(X^i_t)-K^i_t = \int_0^t \left(\mathcal LV(X^i_s) +\omega K^i_s-\omega C_1\right) \dd s + \sum_{n,\tau_n^i\leqslant t} V(X^i_{\tau_n^i})-V(X^i_{\tau_n^i-}),\]
	where $(\tau_n^i)$ are the death times of particle number $i$ as defined in~\eqref{eqf:deftauni}. Moreover, $K^i$ is a continuous process, and $V(X^i)$ is continuous between death times. Hence, $V(X^i)-K^i$ is $\mathcal C^1$ between death times. 
	Let $f(x,y)=((x-y)_+)^2$, so that $f$ is differentiable, and is non-decreasing in the variable $x$. By construction of $V$, for all $n\in\N$, $V(X^i_{\tau_n^i})\leqslant V(X^i_{\tau_n^i-})$, hence we have for all $t\geqslant 0$:
	\[\left(\left(V(X^i_{t})-K^i_{t}\right)_+\right)^2 \leqslant 2\int_0^{t} \omega\left(K^i_s-V\left(X^i_{s}\right)\right)\left(V\left(X^i_{s}\right)-K^i_s\right)_+\dd s \leqslant 0.\]
	Thus, almost surely, $V\left(X^i_{t}\right)$ is bounded by $K^i_t$ for all $t\geqslant 0$, and we are left to show that with high probability, there are only a few $K^i$'s which are greater then $3C_1$ at time $t_\varepsilon$. Write:
	\[G^i_t = \sqrt{2\varepsilon}\int_0^t e^{\omega (s-t)}\na V(X^i_s)\cdot \dd B^i_s. \]Fix some family of indexes $(i_1,\dots,i_k)\in \{1,\dots,N\}^k$.  The $G^i$'s are $L^2$-martingales, hence for any $\xi \in\R$, $\xi\sum_{j=1}^k G^{i_j}$ is a $L^2$-martingale, and:
	\[\exp \left( \xi\sum_{j=1}^k G^{i_j} -  \xi^2 \left\langle \sum_{j=1}^k G^{i_j} \right\rangle\right)\]is a local-martingale. We have that $\left\langle G^i,G^j \right\rangle = 0$ for all $i\neq j$ because the Brownian motions are independent, hence \[\left\langle \sum_{j=1}^k G^{i_j} \right\rangle_t = \sum_{j=1}^k 2\varepsilon \int_0^t e^{2\omega (s-t)}|\na V(X^i_s)|^2 \dd s \leqslant \frac{\varepsilon k \|\na V\|_{\infty}}{\omega} ,\]and using Fatou's Lemma: \[\E\left(\exp \left( \xi\sum_{j=1}^k G^{i_j}_t\right) \right)\leqslant \exp\left(\frac{\varepsilon \xi^2 k \|\na V\|_{\infty}}{\omega}\right),\]for all $t\geqslant 0$.
	Now we can write, using the Markov inequality:
	\[\mathbb P\left( G^{i_1}_{t_\varepsilon} > C_1,\cdots,G^{i_k}_{t_\varepsilon} > C_1 \right) \leqslant \mathbb P\left( \exp \left( \xi\sum_{j=1}^k G^{i_j}\right) > e^{\xi k C_1}   \right) \leqslant e^{-\xi k C_1}e^{\frac{\varepsilon \xi^2 k\|\na V\|_{\infty}}{\omega}}. \]
	Taking $\xi = C_1 \omega/(2\varepsilon \|\na V\|_{\infty})$, one gets: 
	\[\mathbb P\left( G^{i_1}_{t_\varepsilon} > C_1,\cdots,G^{i_k}_{t_\varepsilon} > C_1 \right) \leqslant \exp\left( -C_1^2 \omega/2\varepsilon \|\na V\|_{\infty}  \right)^k =: \tilde q_\varepsilon^k.\]
    We chose $\varepsilon_0$ small enough so that:\[V_0e^{- \omega t_{\varepsilon_0}} + C_1(1-e^{-\omega t_{\varepsilon_0}})<2C_1.\]
	For all $1\leqslant i\leqslant N$, we then have:
	\[ \left\{K^i_{t_\varepsilon}>3C_1\right\}\subset \left\{ G^i_{t_\varepsilon}> C_1 \right\},\]
	and we have for all family of indexes $(i_1,\dots,i_k)$:
	\begin{equation*}
	 \mathbb{P}\left(X^{i_1}_{t_\varepsilon}\in \mathfrak B,\dots,X^{i_k}_{t_\varepsilon}\in \mathfrak B\right) \leqslant \mathbb{P}\left(G^{i_1}_{t_\varepsilon}>C_1,\dots,G^{i_k}_{t_\varepsilon}>C_1\right)\leqslant \tilde q_\varepsilon^k.
	\end{equation*}
	Finally, we conclude with:
	\begin{align*}
		\mathbb{P}_x \left(A\left(\textbf X_{t_\varepsilon}\right)>\alpha N\right) &\leqslant \mathbb{P}\left(\text{There exist at least }\alpha N\text{ indexes }i\text{ such that } X^i_{t_\varepsilon}\in \mathfrak B\right) \\ &\leqslant \sum_{\alpha N\leqslant k\leqslant N}\binom{n}{k} \tilde q_\varepsilon^k \\ &\leqslant \left(2\left(\tilde q_\varepsilon\right)^{\alpha}\right)^N =: q_\varepsilon^N.
	\end{align*}

\end{proof}

\section{Proofs of the main theorems}\label{conclu}

Our goal is to construct a coupling of $\delta_{\mathbf{x}}P^{N,\varepsilon}$ and $\delta_{\mathbf{y}}P^{N,\varepsilon}$ for all $\mathbf{x},\mathbf{y}\in D^N$ in such a way that some distance $\mathbf{d}(\mathbf{x},\mathbf{y})$ is contracted on average by this coupling along time. The basic idea of the coupling is the following: particles are coupled by pair, namely we want the particle $i$ of the system starting at $\mathbf{x}$ to merge, after a time $t_\varepsilon$, with the particle $i$ of the system starting at $\mathbf{y}$. However, contrary to the case of independent particles, here, even if two particles start at the same position (namely $x_i=y_i$), they have a positive probability to decouple before time $t_\varepsilon$. This can be particularly bad for some initial conditions: for instance if most of the pairs start merged but close to the boundary while a decoupled pair is in the middle of the domain, then this will typically lead to a lot of decoupling as coupled pairs rebirth on the uncoupled pair. This will be tackled through the definition of the distance $\mathbf{d}$.

\subsection{Long time convergence}

We now construct the coupling of $(\delta_{\mathbf{x}} P^N_{t})_{t\geqslant 0}$ and $(\delta_{\mathbf{y}} P^N_{t})_{t\geqslant 0}$ for all $\textbf x,\textbf y\in D^N$, that will yield a bound on the distance between $\delta_{\mathbf{x}} P^N_{t_\varepsilon}$ and $\delta_{\mathbf{y}} P^N_{t_\varepsilon}$. Fix $\textbf x,\textbf y\in D^N$, and a sequence $(I^{i}_n)$ of independent random variable, where $I_n^i$ is uniform on $\cco 1,N \ccf\setminus\left\{i\right\}$.

For all $1\leqslant i \leqslant N$, consider a  coupling $(\tilde X^i_t,\tilde Y^i_t)$ of the diffusion~\eqref{descente2} starting from $(x_i,y_i)$ such as the one constructed in Lemma~\ref{coupopt} (with these processes being independent for two different values of the index $i$). Recall that $\T^d= (\R/2L\mathbb Z)^d$, and $L$ is great enough so that we may consider that $D\subset \T^d$. Hence we may write \[\tilde \tau_1=\inf\left\{ t\geqslant 0, \exists i\in\cco 1,N\ccf, \tilde X^i_t \notin D \text{ or }\tilde Y^i_t \notin D \right\}.\]
Denote by $i_1$ the index of the particles that exit the domain at time $\tilde \tau_1$. For all $i\neq i_1$ and $0\leqslant t \leqslant \tilde \tau_1$ or $i=i_1$ and $0\leqslant t <\tilde \tau_1$, let: 
\[X^i_t=\tilde X^i_t \qquad \text{and} \qquad Y_t^i= \tilde Y_t^i, \]
in the sense that $X^i_t$ (resp. $Y^i_t$) is the only point of $D$ whose projection is $\tilde X^i_t$ (resp. $\tilde Y^i_t$).
Finally, if $\tilde X^{i_1}_{\tilde\tau_1}\notin D$, then set $X^{i_1}_{\tilde\tau_1} = X^{I_1^{i_1}}_{\tilde\tau_1}$, else set $X^{i_1}_{\tilde\tau_1}=\tilde X^{i_1}_{\tilde\tau_1}$. The same goes for $Y^{i_1}_{\tilde \tau_1}$: if $\tilde Y^{i_1}_{\tilde\tau_1}\notin D$, then set $Y^{i_1}_{\tilde\tau_1} = Y^{I_1^{i_1}}_{\tilde\tau_1}$, else set $Y^{i_1}_{\tilde\tau_1}=\tilde Y^{i_1}_{\tilde\tau_1}$.
The coupling can then be constructed by induction, just as for the construction of the FV processes in the introduction.

Basically, the coupling is as follow: two particles with same index will be an optimal coupling of the diffusion as long as they don't die as constructed in Lemma~\ref{coupopt}, and if they die while coupled, then they resurrect using the same index. By this we mean that the uniform variable $I_n^i$ used in the construction of the process in Section~\ref{main} must be the same for both systems. 

We will show that this coupling yields a contraction for the Wasserstein distance associated to a particular distance on $D^N$, namely:
\begin{equation}\label{distance}
\mathbf d(\mathbf{x},\mathbf{y}) = \sum_{i=1}^N\mathbbm{1}_{x_i\neq y_i}\left(1 + \beta V(x_i) + \beta V(y_i)\right) +   \left(1+V_0\right)N\left(\mathbbm{1}_{A(\mathbf{x})>\alpha N}+\mathbbm{1}_{A(\mathbf{y})>\alpha N}\right)\mathbbm{1}_{\mathbf{x}\neq \mathbf{y}},
\end{equation}
where $\beta,\alpha>0$ are parameters that will be chosen small enough, and $A(\mathbf{x})$ has been defined in \eqref{numbbound}. We define as well:
\[d^1(x_i,y_i)=\mathbbm{1}_{x_i\neq y_i}\left(1 + \beta V(x_i) + \beta V(y_i) \right).\] 

The meaning of this distance, which follows the construction of Hairer and Mattingly in \cite{HairerMattingly2008}, is this: if $x_i\neq y_i$ and $V(x_i)+V(y_i) <C_2$, where $C_2$ is as in Proposition~\ref{prop:couplage}, then both particles of index $i$ are in the center of the domain at initial time, and we are able to couple $X^i$ and $Y^i$ before time $t_\varepsilon$ and before they die with high probability. If $x_i\neq y_i$ and $V(x_i)+V(y_i) \geqslant C_2$, then we may not be able to couple them, but the Lyapunov functional will decrease on average. In any case, if $x_i\neq y_i$, $\E\left(d^1(X^i_t,Y^i_t)\right)$ will decrease between initial time and time $t_\varepsilon$. If $x_i=y_i$, then we cannot expect any contraction of $\E\left(d^1(X^i_t,Y^i_t)\right)$, since it is equal to zero at initial time, and the probability that $X^i$ and $Y^i$ decouple is positive (if they die and resurrect on an uncoupled pair). In this case, if $x_i$ is in the center of the domain, then, as we will see below, the probability of decoupling is very small and won't be an issue. But in the case where there are many particles coupled at $t=0$ close to the boundary, many of them will get separated. This is why we added the additional term  $N\left(\mathbbm{1}_{A(\mathbf{x})>\alpha N}+\mathbbm{1}_{A(\mathbf{y})>\alpha N}\right)\mathbbm{1}_{\mathbf{x}\neq \mathbf{y}}$ in the definition of $\mathbf{d}$. If we are in this case, this term is initially not   zero but, according to Lemma~\ref{boundary}, it will probably be zero at time $t_\varepsilon$, which will compensate for the non-zero terms that will appear with other parts of the distance. In other words, this term plays the role of a global Lyapunov function, by contrast with the pairwise Lyapunov function $V(x_i)+V(y_i)$.

Let's start by bounding from above the probability to decouple. This is the part where we use Assumption~\ref{assu2}. Recall that $\mathfrak B=\{V>3C_1\}$.

\begin{lem}\label{deathprobability}
	Under Assumptions~\ref{assu1} and \ref{assu2}, there exists $C_3$ such that for all $0<\alpha<1/4$, there exists $\varepsilon_0>0$ such that for all $0<\varepsilon<\varepsilon_0$, there exists $m_\varepsilon>0$, such that for all $N\in\mathbb{N}$, $\mathbf{x},\mathbf{y}\in D^N$ with $A(\mathbf{x}),A(\mathbf{y})\leqslant \alpha N$, and all $i\in\cco 1,N\ccf$ such that $x_i=y_i$, we have:
	\[\mathbb{P}_{\mathbf{x},\mathbf{y}}\left(X^i_{t_\varepsilon}\neq Y^i_{t_\varepsilon}\right)\leqslant \left\{ \begin{matrix}
		m_\varepsilon C_3\bar d(\mathbf{x},\mathbf{y})/N &\text{ if }x_i \notin  \mathfrak B \\ C_3\bar d(\mathbf{x},\mathbf{y})/N &\text{ if }x_i\in \mathfrak B 
	\end{matrix} \right.\]where $\bar d(\mathbf{x},\mathbf{y}) = \sum_{i=1}^N \mathbbm{1}_{x_i\neq y_i}$, and $\lim_{\varepsilon\rightarrow0}m_\varepsilon =0$.
\end{lem}

An intermediate lemma is needed. The goal of this lemma is to get bounds on the number of death events.

\begin{lem}\label{deathnumber}
Under Assumptions~\ref{assu1} and~\ref{assu2}, let $\mathfrak B_1$ be the neighborhood of $\partial D$ fixed in Section~\ref{Lyapunov}.
Write the event:
\[\mathcal A=\left\{\#\left\{i\in \cco 1,N\ccf, \exists t\leqslant t_\varepsilon, X^i_t\in \mathfrak B_1 \right\} \geqslant 2\alpha N\right\}.\]
\begin{enumerate}
\item There exists $\varepsilon_0>0$ such that for all $0<\varepsilon<\varepsilon_0$, there exists $\tilde p_\varepsilon>0$, such that for all $0<\alpha<1/4$, $N\in\N$, $\mathbf{x}\in D^N$ with $A(\mathbf x)\leqslant \alpha N$, 
\[\mathbb P_{\mathbf{ x}}(\mathcal A)\leqslant  \left(2\tilde p_{\varepsilon}^{\alpha}\right)^N,\]and $\lim_{\varepsilon\rightarrow0}\tilde p_\varepsilon=0$.
\item Moreover, if $T$ denote the number of rebirth in the system before time $t_\varepsilon$, there exists $\varepsilon_0,\sigma>0$, $0<q<1$, such that for all $0<\varepsilon<\varepsilon_0$ and $0<\alpha<1/4$:
\[\mathbb P(T>\sigma N,\mathcal A^c)\leqslant q^N.\]

\item Write $T^i$ the number of rebirth of particle $i$ before time $t_\varepsilon$. We have as well that there exist $C,\varepsilon_0>0$ such that for all $0<\varepsilon<\varepsilon_0$, for all $0<\alpha<1/4$, $\mathbf{x}\in D^N$ satisfying $A(\mathbf{x}) <\alpha N$:
\[\E_{\mathbf{x}}\left( \po T^i\pf^{2} \mathbbm{1}_{\mathcal A^c}\right) \leqslant \left\{ \begin{matrix} C\bar p_\varepsilon\text{ if }x_i\notin \mathfrak B \\ C\text{ if }x_i\in \mathfrak B , \end{matrix} \right. \]
\end{enumerate}
where $\bar p_\varepsilon$ is given in ~\eqref{strongdeathprob}.
\end{lem}

\begin{proof}

\begin{enumerate}
\item
At time $t=0$, the condition on $\mathbf{x}$ implies that there are less than $\alpha N$ particles in $\mathfrak B$. Under Assumption~\ref{assu2}, we noticed in Remark~\ref{remarque} that $\mathfrak B_1\subset \left\{V>m\right\}$, with $m >3C_1$, so that $\mathfrak B_1\subset \mathfrak B$. This means that for $\mathcal A$ to happen (namely for $2\alpha N$ particles to visit $\mathfrak B_1$ before time $t_\varepsilon$), at least $\alpha N$ particles that were initially  in  $D\setminus \mathfrak B$ must have reached $\mathfrak B_1$ before time $t_{\varepsilon}$. Write:
	\[\widetilde p_\varepsilon = \sup_{x\in D\setminus \mathfrak B}\mathbb{P}_x\left(\tau_{\mathfrak B_1}<t_\varepsilon \right),\]
	where $\tau_{\mathfrak B_1}$ is the first hitting time of the set $\mathfrak B_1$ for the diffusion~\eqref{descente}. Recall from~\eqref{eq:condB1} that $\mathfrak B_1$ satisfies $a< \inf_{\mathfrak B_1}U$. Together with the fact that $\min_{\mathfrak B_1} V > m > 3C_1 = \max_{D\setminus \mathfrak B} V$ (so that  $\mathfrak B$ is a neighborhood of $\overline{\mathfrak B}_1\cap D$, hence the distance between $\mathfrak B_1$ and $D\setminus \mathfrak B$ is positive), this implies  that $\tilde p_\varepsilon\rightarrow 0$ as $\varepsilon\rightarrow 0$  thanks to \cite[Chapter 6, Theorem 6.2]{Freidlin-Wentzell}. The fact that a particle reaches $\mathfrak B_1$ only depends on the Brownian motion driving it, hence we have:
	\[\mathbb P(\mathcal A)\leqslant \sum_{k\geqslant \alpha N} \binom{(1-\alpha)N}{k}\tilde p_\varepsilon^k \leqslant (2\tilde p^{\alpha}_{\varepsilon})^N.\]

\item In order to control the number of deaths of the $i^{th}$ particle up to time $t_\varepsilon$, we are going to distinguish two types of rebirth events: either the particle is resurrected on a particle which we know never reaches $\mathfrak B_1$ (i.e. stays away from the boundary), in which case we can bound the probability that the $i^{th}$ particle dies again, or it is resurrected on a particle for which we have no information, in which case it can be arbitrarily close to the boundary and the time before the next death of the $i^{th}$ particle can be arbitrarily small.

For convenience, we consider in the rest of the proof that the FV process has been defined thanks to a construction similar to the one presented in Section~\ref{main} except that the Brownian motions driving the SDEs are changed at each death event, namely along with the variables $(I_n^i)_{n\geqslant 0,i\in\cco 1,N\ccf}$, we consider a family of independent $d$-dimensional Brownian motions $((B^{n,i}_t)_{t\geqslant0})_{n\geqslant 0,i\in\cco 1,N\ccf}$, so that after its $n^{th}$ death and up to its $(n+1)^{th}$ death the position of the particle $i$ is given by $X_{\tau_n^i+t}^i = \bar X_t^i$ where $\bar X^i$ is the solution of \eqref{descente} driven by $B^{n,i}$ with initial condition $\bar X_0^i = X_{\tau_n^i}^i$ (recall the notation $\tau_n^i$ from \eqref{eqf:deftauni}). Of course the law of the process is correct with this construction.

Denote:
\[\textbf S=\left\{i\in \cco 1,N\ccf, \exists t<t_\varepsilon, X_t^i \in \mathfrak B_1\right\}.\]

Then the Markov inequality yields:
\begin{align*}
     \mathbb P\left( T >\sigma N,\mathcal A^c\right) \leqslant e^{-\sigma N}\E\left(e^{T}\mathbbm{1}_{\mathcal A^c}\right) = e^{-\sigma N}\sum_{\underset{\#S\leqslant 2\alpha N}{S\in \mathcal P(\cco 1,N\ccf)}}\E\left(e^{\sum_{i=1}^N T_i}\mathbbm{1}_{\textbf S=S}\right).
\end{align*}

Fix $S\in \mathcal P(\cco 1,N\ccf)$, such that $\#S<2\alpha N$, and recall the definition of the variable $I^i_n$ used in the construction of the FV process, which are independent uniform variables on $\cco 1,N\ccf \setminus\left\{i\right\}$. We define by induction $P^i_0=0$ and :
\[P^i_k = \inf\left\{n>P_{k-1}^i, I_n^i\notin S\right\}.\]
Notice that, under the event $\{\mathbf{S}=S\}$, if $I_n^i\notin S$, it means that at its $n^{th}$ rebirth the particle $i$ is resurrected on a particle which never reaches $\mathfrak B_1$ before time $t_\varepsilon$.

Setting $k_0(i)=1$ if $x_i \in \mathfrak B_1$ and $k_0(i)=0$ otherwise, we  define as well 
\[P^i=\inf\left\{k \geqslant k_0(i), \forall x\in D\setminus \mathfrak B_1, \tau_D(X^{x,i,P_k^i}) >t_\varepsilon \right\},\]
where for $n\in\mathbb N$ the family of processes $(X^{x,i,n})_{x\in D\setminus \mathfrak B_1}$ is as in Assumption~\ref{assu2} and are driven by the Brownian motion $B^{n,i}$. Since we have already observed that, for all $k>0$, at its $(P_k^i)^{th}$ death, the particle $i$ is resurrected at a position in $D\setminus \mathfrak B_1$,
 the event $\{\forall x\in D\setminus \mathfrak B_1, \tau_D(X^{x,i,P_k^i}) >t_\varepsilon\}$, which is measurable with respect to the Brownian motion $B^{P_k^i,i}$, implies that the particle does not die again before time $t_\varepsilon$. For $k=0$, it depends whether initially $x_i \in \mathfrak B_1$: if $x_i \notin \mathfrak B_1$ (which is in particular the case if $x_i\notin \mathfrak B$) then, again, the event $\{\forall x\in D\setminus \mathfrak B_1, \tau_D(X^{x,i,0}) >t_\varepsilon\}$ implies that the particle doesn't die before time $t_\varepsilon$. This is not the case if $x_i\in \mathfrak B_1$. As a consequence, in any cases, under the event $\{\mathbf{S}=S\}$, we can bound the total number of death of the $i^{th}$ particle by
 \[T^i \leqslant \sum_{k=1}^{P^i} (P_k^i - P_{k-1}^i).\]

The variables $(P^i_{k}-P^i_{k-1})_{k\geqslant 1,i}$ are independent geometric random variables of parameter   $1-\#S>1-2\alpha$. Under Assumption~\ref{assu2},  if $x_i \in \mathfrak B_1$ (resp. if $x_i\notin \mathfrak B_1$) then $P^i$ (resp. $P^i+1$) is a  geometric random variable of parameter  $1-\bar p_\varepsilon$. Moreover, $(P^i)_{1\leqslant i\leqslant N}$ is a family of independent random variables, independent from $(P_k^i)_{k\geqslant 1, 1\leqslant i\leqslant N}$. We have:
\[\E\left(e^{\sum_{i=1}^N T_i}\mathbbm{1}_{\mathbf S=S}\right) \leqslant \E\left( e^{\sum_{i=1}^N \sum_{k=1}^{P^i}(P^i_k-P^i_{k-1})} \right) = \left( \E\left( e^{\sum_{k=1}^{P^i}(P^i_k-P^i_{k-1})}\right)\right)^N.\]
We are just left to show that $\E\left( \exp{\sum_{k=1}^{P^i}(P^i_k-P^i_{k-1})}\right)$ is finite and bounded uniformly in $\varepsilon<\varepsilon_0$. Conditioning with respect to $P^i$ we get :
\begin{align*}
\E\left( e^{\sum_{k=1}^{P^i}P^i_k-P^i_{k-1}}\right) = \E \left( \E\left( e^{P^i_0} \right)^{P^i}\right) \leqslant \E\left( \left( \frac{e}{1-e\alpha}\right)^{P^i}\right), 
\end{align*}
hence the result if $\varepsilon_0$ satisfies $\bar p_{\varepsilon_0} < \frac{1-e\alpha}{e}$, since we bound then
\begin{align*}
     \mathbb P\left( T >\sigma N,\mathcal A^c\right) \leqslant \left( 2 e^{-\sigma } \E\left( \left( \frac{e}{1-e\alpha}\right)^{P^i}\right)\right)^N.
\end{align*}

\item In the same spirit, fix $i\in \cco 1,N\ccf$, 
and write now: 
\[P^i_k = \inf\left\{n>P_{k-1}^i, I_n^i\notin \textbf S^i\right\},\]
where
\[\textbf S^i=\left\{j\in \cco 1,N\ccf\setminus \left\{i\right\}, \exists t<t_\varepsilon, X_t^{j} \in \mathfrak B_1\right\}
,\]
and 
\[\mathcal A^i=\left\{\#\left\{j\in \cco 1,N\ccf\setminus \left\{i\right\}, \exists t\leqslant t_\varepsilon, X^{j}_t\in \mathfrak B_1 \right\} \geqslant 2\alpha N\right\},\]
and the definition of $P^i$ does not change.
We have that $(P^i_k-P^i_{k-1})_k$, and $P^i$ are independent random variable, and $P^i$ is independent of $\mathcal A^i$ and $\textbf S^i$. Indeed, $\mathcal A^i$ and $\textbf S^i$ only depends on the Brownian motions that drive $(X^j)_{j\neq i}$. Under the event $(\mathcal A^i)^c$, the cardinality of $\mathbf{S}^i$ is less than $2\alpha N$. Furthermore, we have that $\mathcal A^i\subset\mathcal A$, and hence, 
as in the previous step, using Cauchy-Schwarz inequality and that the second moment of  a geometric variable with  parameter $q$ is less than $2/q^2$,
\begin{multline*}
\E\left(\po T^i\pf^2 \mathbbm{1}_{\mathcal A^c}\right) \leqslant \E\po \left( \sum_{k=1}^{P^i}\left(P_k^i -P_{k-1}^i\right) \mathbbm{1}_{(\mathcal A^i)^c} \right)^2\pf \ \\ \leqslant \  \E\left( \po P^i\pf^2 \E\po\po P_1^i\pf^2|\textbf S^{i},\mathcal A^i,P^i\pf\right)  \leqslant \frac{2}{\po 1-2\alpha\pf^2}\E\po \po P^i\pf^2 \pf,
\end{multline*}
and we conclude by bounding $\E\po \po P^i\pf^2 \pf \leqslant 2\bar p_\varepsilon (1-\bar p_\varepsilon)^{-2}$ 
if $x_i\notin \mathfrak B$ (since, then, $x_i\notin \mathfrak B_1$, so that $P^i+1$ is a geometric variable with parameter $1-\bar p_\varepsilon$) and, otherwise, $\E\po \po P^i\pf^2 \pf \leqslant  2(1-\bar p_\varepsilon)^{-2}$.

\end{enumerate}

\end{proof}

\begin{proof}[Proof of Lemma~\ref{deathprobability}]
    Define the sets:
    \begin{align*}
		&U_1(0) = \left\{i\in\cco 1,N\ccf,\  x_i\neq y_i\right\},\\
		&U_2(0) = \left\{i\in\cco 1,N\ccf,\  x_i=y_i\right\}.
	\end{align*}
	Now, for $t\geqslant 0$, we want to define some sets $U_1(t)$, $U_2(t)$, such that if $X^i$ and $Y^i$ decouple at some time $s\geqslant 0$, then for all $t\geqslant s$, $i\in U_1(t)$. For $i\in U_2(0)$, $n\in\N$, write:
	\[\tau^i_n=\inf\left\{ t >  \tau_{n-1}^i, X_{t-}^i=Y^i_{t-}\in \partial D \right\},\]
	as in\eqref{eqf:deftauni}, and \[\bar \tau^i_d = \inf\left\{t\geqslant 0, X^i_t\neq Y_t^i\right\}.\]
	Since the FV-process is well-defined, almost surely, there is only a finite number of such events before time $t_\varepsilon$, for all $1\leqslant i\leqslant N$. Then define the set $U_1(t)$ and $U_2(t)$ for $t\in (\tau_{k-1},\tau_k]$ by induction on $k\geqslant 1$. Assume that the sets have been defined up to the time $ \tau_{k-1}$ for some $k\geqslant 1$. Set $U_j(t)=U_j( \tau_{k-1})$ for all $t\in( \tau_{k-1}, \tau_k)$. Let $i\in\cco 1,N\ccf$ be the index such that $ \tau_k\in \cup_{n\in\N}\{\bar \tau^i_n\}$. Now we distinguish two cases. If $ \tau_k\neq \bar\tau^i_d$, then $U_j(\tau_k)=U_j(\tau_{k-1})$ for $j=1,2$. Else set:
	\[U_1(\tau_k) = U_1(\tau_{k-1}) \cup \{i\}\,,\quad U_2(\tau_k) = U_2(\tau_{k-1}) \cap \{i\}^c \,.\]
	It is immediate to check that $U_1(t)$ and $U_2(t)$ form a partition of $\cco 1,N\ccf$ for all $t\geqslant 0$, and that $U_1(t)$ is non-decreasing with $t$ and such that $\{i\in\cco 1,N\ccf,\ X_t^i \neq Y_t^i\} \subset U_1(t)$ for all $t\geqslant 0$.
Recall from Lemma~\ref{deathnumber} the event: 
	\[\mathcal A=\left\{\#\left\{i\in \cco 1,N\ccf, \exists t\leqslant t_\varepsilon, X^i_t=Y_t^i\in \mathfrak B_1 \right\} \geqslant 2\alpha N\right\}.\]
	For $n\in\N$ and $j=1,2$, write $u_k^j = \# U_j(\tau_k)$. At each time $\tau_{k+1}$, the probability that a particle goes from $U_2$ to $U_1$ is less than $u^1_{k}/N$. Hence, we have that for all $k\geqslant 1$:
    \[\E\left(\po u^1_{k+1}\pf^2|\mathcal F_{\tau_{k}}\right) = \po u^1_k + B_k\pf^2, \]
    where $B_k$ is a Bernoulli random variable with parameter bounded by $u^1_{k}/N$. Therefor,
	\[\E\left(\po u^1_{k+1}\pf^2|\mathcal F_{\tau_{k}}\right)  \leqslant \po u^1_k\pf^2\left(1+\frac{3}{N}\right),\]
	and thus
	\[\E\left(\po u^1_k\pf^2\right) \leqslant \bar d(x,y)^2\left(1+\frac{3}{N}\right)^k.\]
    Using the notations of Lemma~\ref{deathnumber}, in particular $T$ to denote the total number of death event before time $t_\varepsilon$,
    using that $u^1_n$ is non-decreasing, we bound
    \begin{align*}
        \E\po \po u_T^1\pf^2\mathbbm{1}_{\mathcal A^c}\pf &\leqslant \E \po \po u_{\sigma N}^1\pf^2\pf + N^2\mathbb P(T>\sigma N,\mathcal A^c)   \\& \leqslant e^{3\sigma} \bar d(x,y)^2 + N^2q^N,
    \end{align*}
    which is bounded uniformly on $N\geqslant 1$ and $\varepsilon$ small enough, by $\tilde C_3\bar d(x,y)^2$, for some $\tilde C_3>0$, as soon as $\bar d(x,y)\geqslant 1$ (while, if $\bar d(x,y)=0$ then the two processes remain equal for all times and thus the result is trivial). 
    We get from all of this: 
	\[\E\left(\sup_{t\leqslant t_\varepsilon} \bar d\left(\textbf X_{t},\textbf Y_{t}\right)^2\mathbbm{1}_{\mathcal{A}^c}\right)\leqslant  \E\po \po u^1_T\pf^2\mathbbm{1}_{\mathcal{A}^c}\pf \leqslant \tilde C_3\bar d(x,y)^2.\]
	Now we can bound the probability to decouple starting from any $x_i=y_i\in D$, for a fixed $i$ (recall the notation $T^i$ from Lemma~\ref{deathnumber}):
	\begin{align*}
	    \mathbb P\left(\exists 0<t<t_\varepsilon, X^i_t\neq Y_t^i\right) &\leqslant \sum_{n\geqslant 1} \mathbb{P}\left(\bar\tau_d^i=\bar\tau^i_n, \mathcal A^c, T_i>n \right) + \mathbb P(\mathcal A)\\ &= \sum_{n\geqslant 1} \E\left(\bar d(\textbf X_{\tau_n},\textbf Y_{\tau_n})/N\mathbbm{1}_{T_i>n}\mathbbm{1}_{\mathcal A^c}\right)  + \mathbb P(\mathcal A) \\ &\leqslant \frac{1}{N}\E\left(\sup_{t\leqslant t_\varepsilon} \bar d\left(\textbf X_{t},\textbf Y_{t}\right)\mathbbm{1}_{\mathcal A^c}\sum_{n\geqslant 1}\mathbbm{1}_{T_i>n}\right) + (2\tilde p_\varepsilon^{\alpha})^N \\ &= \frac{1}{N}\E\left(\sup_{t\leqslant t_\varepsilon} \bar d\left(\textbf X_{t},\textbf Y_{t}\right)\mathbbm{1}_{\mathcal A^c}T_i\right) + (2\tilde p_\varepsilon^{\alpha})^N \\ &\leqslant \frac{1}{N}\sqrt{\E\left(\sup_{t\leqslant t_\varepsilon} \bar d\left(\textbf X_{t},\textbf Y_{t}\right)^2\mathbbm 1_{\mathcal A^c}\right)} \sqrt{\E(T_i^2\mathbbm 1_{\mathcal A^c})} + (2\tilde p_\varepsilon^{\alpha})^N \\ &\leqslant \sqrt{\tilde C_3}\frac{\bar d(x,y)}{N} \sqrt{\E(T_i^2\mathbbm 1_{\mathcal A^c})} + (2\tilde p_\varepsilon^{\alpha})^N.
	\end{align*}
	We conclude thanks to Lemma~\ref{deathnumber}, and using that, for $\varepsilon$ small enough, $(2\tilde p_\varepsilon^{\alpha})^N \leqslant \bar{d} (x,y)/N$ as soon as $\bar{d}(x,y) \geqslant 1$.
\end{proof}

We need to choose the parameters involved in the definition of the distance $\mathbf d$. There are three of them: $\alpha$, $\beta$, and $V_0$. We fix any $V_0>4C_1$ (which is required in Lemma~\ref{Lyapu}), any  $\beta<(2\vee 4C_1)^{-1}$  and then $\alpha$ small enough so that 
\begin{equation}\label{cond1}
\frac{1+2\beta C_1}{1+\beta C_2}\vee 4\beta C_1 + \alpha C_3(1+2\beta V_0)<1,
\end{equation}
and
\begin{equation}\label{cond2}
\frac{1+2\beta V_0}{1+V_0}<1,
\end{equation}
where we used in \eqref{cond1} that $C_2>2C_1$ from Proposition~\ref{prop:couplage}.
This is possible by fixing first some small $\beta$, and then taking any $V_0>4C_1$, and finally $\alpha$ small enough.

\begin{lem}\label{coupkilled}
	Let $\mathbf{x},\mathbf{y}\in D^N$   and $1\leqslant i\leqslant N$ such that $x_i\neq y_i$ and $V(x_i)+V(y_i)\leqslant C_2$. Then with $\kappa_{1,\varepsilon} = \gamma_\varepsilon\vee \left(1-c_\varepsilon+2p_\varepsilon + 4\beta C_1(1-\gamma_{\varepsilon})\right)$, where $\gamma_\varepsilon$ has been defined in Lemma~\ref{Lyapu}, we have: 
	\[\E(d^1(X^i_{t_\varepsilon},Y^i_{t_\varepsilon}))\leqslant \kappa_{1,\varepsilon} d^1(x_i,y_i).\]
\end{lem}

\begin{proof}
	Let $(\tilde X^{i}_t,\tilde Y_t^{i})$ be the coupling of the diffusion \eqref{descente2} as in lemma \ref{coupopt}, used in the construction of our coupling. Then, $(\tilde X^{i}_t,\tilde Y_t^{i})=(X^i_t,Y_t^i)$ until $X^i$ or $Y^i$ reaches $\partial D$. We have :
	\begin{align*}
		\pro{X^i_{t_\varepsilon} = Y^i_{t_\varepsilon}} &\geqslant \pro{X^i_{t_\varepsilon} = Y^i_{t_\varepsilon},\tau_{x_i}>t_\varepsilon,\tau_{y_i}>t_\varepsilon} \\ &=\pro{\tilde X^{x_i}_{t_\varepsilon} = \tilde Y^{y_i}_{t_\varepsilon},\tau_{x_i}>t_\varepsilon,\tau_{y_i}>t_\varepsilon} \\ &\geqslant \pro{\tilde X^{x_i}_{t_\varepsilon} = \tilde Y^{y_i}_{t_\varepsilon}} - \pro{\tau_{x_i}>t_\varepsilon} - \pro{\tau_{y_i}>t_\varepsilon} \\ &\geqslant c_{\varepsilon} - 2p_\varepsilon.
	\end{align*}
	Using the property of the Lyapunov function described in \eqref{lyapu}, we then have:
	\begin{align*}
		\E\left(d^1(X^i_{t_\varepsilon},Y^i_{t_\varepsilon})\right) &\leqslant 1 - c_{\varepsilon} + 2p_\varepsilon + 2\beta C_1(1-\gamma_{\varepsilon}) + \gamma_\varepsilon \beta \left( V(x_i)+V(y_i) \right)\\
		& \leqslant  \kappa_{1,\varepsilon} d^1(x_i,y_i).
	\end{align*}
\end{proof}

Now we focus on the particles near the boundary that are not coupled:

\begin{lem}\label{decreaselya}
	Let $\mathbf{x},\mathbf{y}\in D^N$ and $1\leqslant i\leqslant N$ such that $V(x_i)+V(y_i)\geqslant C_2$ and $x_i\neq y_i$. Then with $\kappa_{2,\varepsilon} = \gamma_{\varepsilon} + (1-\gamma_{\varepsilon})\frac{1+2\beta C_1}{1+\beta C_2}$, we have: 
	\[\E(d^1(X^i_{t_\varepsilon},Y^i_{t_\varepsilon}))\leqslant \kappa_{2,\varepsilon} d^1(x_i,y_i).\]
\end{lem}

\begin{proof}
	Using the Lyapunov property and the fact that $\gamma_{\varepsilon}\leqslant\kappa_{2,\varepsilon}$, we have:
	\begin{multline*}
		\E\left(d^1(X^i_{t_\varepsilon},Y^i_{t_\varepsilon})\right) \leqslant 1 + 2\beta C_1(1-\gamma_{\varepsilon}) + \beta\gamma_{\varepsilon}\left(V(x_i) + V(y_i) \right) \\ \leqslant \kappa_{2,\varepsilon}d^1(x_i,y_i) +  1 + 2\beta C_1(1-\gamma_{\varepsilon}) - \kappa_{2,\varepsilon} + \beta (\gamma_{\varepsilon}-\kappa_{2,\varepsilon})\left(V(x_i)+V(y_i)\right).
	\end{multline*}
	The fact that $V(x_i)+V(y_i) \geqslant C_2$ implies that 
	\[ 1 + 2\beta C_1(1-\gamma_{\varepsilon}) - \kappa_{2,\varepsilon} + \beta (\gamma_{\varepsilon}-\kappa_{2,\varepsilon})\left(V(x_i)+V(y_i)\right)\leqslant 0,\]and thus the result.
\end{proof}

\begin{proof}[Proof of Theorem~\ref{thm}]
	Let $\mathbf{x},\mathbf{y}\in D^N$, $\kappa_\varepsilon=\kappa_{1,\varepsilon}\vee\kappa_{2,\varepsilon}$. First suppose that $\mathbbm{1}_{A(\mathbf{x})>\alpha N}=\mathbbm{1}_{A(\mathbf{y})>\alpha N}=0$. We decompose:
	\begin{multline}\label{prouf}
		\E\left(\mathbf d\left(\mathbf{X}_{t_\varepsilon},\mathbf{Y}_{t_\varepsilon}\right)\right) \\= \sum_{i/x_i\neq y_i}\E\left(d^1\left(X^i_{t_\varepsilon},Y^i_{t_\varepsilon}\right)\right) + \sum_{i/x_i=y_i \notin \mathfrak B}\E\left(d^1\left(X^i_{t_\varepsilon},Y^i_{t_\varepsilon}\right)\right) + \sum_{i/x_i=y_i \in \mathfrak B}\E\left(d^1\left(X^i_{t_\varepsilon},Y^i_{t_\varepsilon}\right)\right) \\+ N\left(1+V_0\right)\left( \mathbb{P}\left(A(\mathbf X_{t_\varepsilon})>\alpha N\right) + \mathbb{P}\left(A(\mathbf Y_{t_\varepsilon})>\alpha N\right) \right)\,.
	\end{multline}
	Thanks to Lemmas~\ref{coupkilled} and \ref{decreaselya}, we have that the first sum is less than $\kappa_\varepsilon d(\mathbf{x},\mathbf{y})$.  From Lemma~\ref{deathprobability}, the second term is less than: \[C_3m_{\varepsilon}(1+2\beta V_0)\mathbf d(\mathbf{x},\mathbf{y}),\]
	and the third term is less than:
	\[\alpha C_3(1+2\beta V_0)\mathbf d(\mathbf{x},\mathbf{y}).\]
	Finally, thanks to Lemma~\ref{boundary}, the last term is less than: \[2N\left(1+V_0\right)q_{\varepsilon}^N\mathbf{d}(\mathbf{x},\mathbf{y})\leqslant \frac{-2\left(1+ V_0\right)}{e\ln(q_\varepsilon)}\mathbf d(\mathbf{x},\mathbf{y}).\]
	Putting all of this together we get:
	\[\E(\mathbf d(\mathbf{X}_{t_\varepsilon},\mathbf{Y}_{t_\varepsilon}))\leqslant s_\varepsilon \mathbf d(\mathbf{x},\mathbf{y})\]where \[s_{\varepsilon}=\kappa_\varepsilon + C_3\left(1+2\beta V_0\right)m_\varepsilon + \alpha C_3(1+2\beta V_0)+ \frac{-2\left(1+ V_0\right)}{e\ln(q_\varepsilon)}.\]As $\varepsilon$ goes to $0$, $s_\varepsilon$ goes to $\frac{1+2\beta C_1}{1+\beta C_2}\vee 4\beta C_1 + \alpha C_3(1+2\beta V_0)<1$ because of our choice of constants \eqref{cond1}. \par 
	Now, consider the case where $\mathbbm{1}_{A(\mathbf{x})>\alpha N}+\mathbbm{1}_{A(\mathbf{y})>\alpha N}>0$. Assume that $\mathbf{x}\neq\mathbf{y}$, the result being trivial otherwise since the processes stay equal for all times. In that case, $\mathbf d(\mathbf{x},\mathbf{y}) \geqslant N(1+ V_0)$ and we simply bound
	\[\E\left(d\left(\mathbf{X}_{t_\varepsilon},\mathbf{Y}_{t_\varepsilon}\right)\right) \leqslant N(1+2\beta V_0 + (1+V_0)q_\varepsilon)\leqslant \left(\frac{1+2\beta V_0}{1+V_0}+q_\varepsilon\right)\mathbf d(\mathbf{x},\mathbf{y}),\]
	for $\varepsilon$ small enough. 
	Since $r_\varepsilon := \frac{1+2\beta V_0}{1+V_0} + q_\varepsilon$ is strictly less than $1$ with our choice of constants \eqref{cond2} for $\varepsilon$ small enough, we conclude that, with $c=1-\sup_{\varepsilon<\varepsilon_0}s_\varepsilon\wedge r_\varepsilon>0$ where $\varepsilon_0$ is small enough, we have for all $\mathbf{x},\mathbf{y}\in D^N$:
	\[\E_{\mathbf{x},\mathbf{y}}\left(\mathbf d\left(\mathbf{X}_{t_\varepsilon},\mathbf{Y}_{t_\varepsilon}\right)\right) \leqslant (1-c) \mathbf d(\mathbf{x},\mathbf{y}).\]
	By conditioning with respect to the initial condition we get that:
	\[W_{\mathbf d}\left(\mu P_{t_\varepsilon}^N , \nu P_{t_\varepsilon}^N\right) \leqslant ( 1-c )W_{\mathbf d}\left(\mu ,  \nu \right),\]
	for all probability measures $\mu,\nu$ in $\mathcal{M}^1(D^N)$, and by iteration:
	\begin{align*}W_{\mathbf d}\left(\mu P_{t}^N , \nu P_{t}^N\right)&\leqslant (1-c)^{\left\lfloor t/t_\varepsilon \right\rfloor} W_{\mathbf d}\left(\mu P_{t-\left\lfloor t/ t_\varepsilon\right\rfloor t_\varepsilon}^N , \nu P_{t-\left\lfloor t/ t_\varepsilon\right\rfloor t_\varepsilon}^N\right) \\ &\leqslant (1 + 2(\beta+1)V_0)N (1-c)^{\left\lfloor t/t_\varepsilon \right\rfloor}.
	\end{align*}
	We conclude the first point of the theorem using $\mathbbm{1}_{\mathbf{x}\neq \mathbf{y}} \leqslant \mathbf d(\mathbf{x},\mathbf{y})$.

	Let's now prove the second point of Theorem~\ref{thm}. $\mathcal{M}^1\left(D^N\right)$ endowed with the distance $\mathbf d$ is a complete space. The contraction of $P_{t_\varepsilon}^N$ yield the existence and uniqueness of the stationary measure $\nu^{N,\varepsilon}_{\infty}$, as well as the exponential convergence of $\mu P_t^{N,\varepsilon}$ towards $\nu^{N,\varepsilon}_{\infty}$. If $\mu$ is exchangeable, then $\mu P_t^{N,\varepsilon}$ is exchangeable for all $t\geqslant 0$. The convergence of $\mu P_t^{N,\varepsilon}$ toward $\nu^{N,\varepsilon}_{\infty}$ implies that $\nu^{N,\varepsilon}_{\infty}$ is exchangeable. Now consider an optimal coupling for the distance $d$ of $\mu P_t^{N,\varepsilon}$ and $\nu^{N,\varepsilon}_{\infty} P_t^{N,\varepsilon}$. Using the exchangeability property we have:
	\begin{align*}
	\|\mu P^{N,k}_{t}-\nu^{N,\varepsilon,k}_{\infty}\|_{TV} &\leqslant \E\left(\mathbbm{1}_{\left(X_t^1,\dots,X^k_t\right)\neq \left(Y_t^1,\dots,Y^k_t\right)}\right) \\ &\leqslant \sum_{i=1}^k\E\left(\mathbbm{1}_{X^i_t\neq Y^i_t}\right) \\ &= k\E\left(\mathbbm{1}_{X^1_t\neq Y^1_t}\right) \\&\leqslant \frac{k}{N}\sum_{i=1}^N\E\left(\mathbbm{1}_{X^i_t\neq Y^i_t}\right),
	\end{align*}
	and we conclude with the first point of the theorem.
\end{proof}

\subsection{Propagation of chaos}\label{proof2}

Recall the definition of the empirical measure $\pi^N$ in \eqref{empiri}.
As said in the introduction, the goal is to get a uniform in time propagation of chaos result. We start from a propagation of chaos result, with a time dependency, from \cite{Villemonais}. Their result reads as follows:

\begin{prop}[\cite{Villemonais}, Theorem 1]\label{propag}
For all $\mu_0\in \mathcal M^1(D^N)$, considering $(X_t^i)_{t\geqslant 0}$  the FV process with initial condition $(X_0^i)$ which is a random variable of law $\mu_0$, and $(X_t)$ the diffusion \eqref{descente}, then, for all bounded $f:D\to\R_+$, all $\varepsilon>0$ and all $t\geqslant 0$:
\begin{multline*}
\E\left(\left|\int_D f\dd\pi^N(\textbf{X}_t) - \E_{\pi^N(\textbf X_0)}\left(f(X_t)|\tau_{\partial D}>t\right) \right|\right)\\
\leqslant\frac{2(1+\sqrt{2})\|f\|_{\infty}}{\sqrt{N}}\sqrt{\E\left(\frac{1}{\mathbb{P}_{\pi^N(\textbf X_0)}\left(\tau_{\partial D}>t\right)^2}\right)},    
\end{multline*}
where $\tau_{\partial D}$ is defined in \eqref{deathtime}.
\end{prop}

We also need a result on the convergence of the law of the diffusion \eqref{descente}, conditioned on survival, towards the QSD $\nu_{\infty}^{\varepsilon}$. This is from \cite{champagnat2018general}, although the statement is slightly modified to fit our setting.

\begin{prop}\label{convnonlin}
Under Assumption~\ref{assu1}, there exists $\varepsilon_0>0$ such that for all compact $K\subset D$, and all $0<\varepsilon<\varepsilon_0$, there exists $C_\varepsilon,\tilde{C}_\varepsilon,\lambda_\varepsilon,\chi_\varepsilon>0$ such that for all $\mu_0\in\mathcal M^1(D)$ satisfying $\mu_0(K)\geqslant 1/4$:
\[\|\mathbb{P}_{\mu_0}\left(X_t\in \cdot|\tau_{\partial D}>t\right)-\nu_{\infty}^{\varepsilon}\|_{TV}\leqslant C_\varepsilon e^{-\chi_\varepsilon t},\]and
\[\mathbb{P}_{\mu_0}\left(\tau_{\partial D}>t\right)\geqslant \tilde C_\varepsilon e^{-\lambda_\varepsilon t}.\]
\end{prop}

\begin{proof}
The process~\eqref{descente} satisfies equation~(4.7) of \cite{champagnat2018general}, with $D_0=F$, where $F$ was defined in the proof of Lemma~\ref{Lyapu}, some $\lambda_1$ independent from $\varepsilon$, and $\varphi=V$. The constant $\lambda_0$ defined in equation~(4.4) of \cite{champagnat2018general} goes to zero as $\varepsilon$ goes to zero, hence we may chose $\varepsilon_0$ such that for all $\varepsilon<\varepsilon_0$, $\lambda_1>\lambda_0$, and assumption of \cite[Corollary 4.3]{champagnat2018general} hold true. From \cite[Theorem 4.1]{champagnat2018general}, this yields the existence of the $QSD$ $\nu_{\infty}^{\varepsilon}$ and of some function $\phi:D\to\R_+^*$, uniformly bounded away from $0$ on all compact subsets of $D$ such that for all $\mu_0\in\mathcal M^1(D)$:
\[
\|\mathbb{P}_{\mu_0}\left(X_t\in \cdot|\tau_{\partial D}>t\right)-\nu_{\infty}^{\varepsilon}\|_{TV}\leqslant C_\varepsilon e^{-\chi_\varepsilon t}\mu_0(V)/\mu_0(\phi).
\]
Since $V$ is bounded,  if $\mu_0(K)\geqslant 1/4$, then $\mu_0(\phi)\geqslant 1/4\min_{K}\phi$, and we get that
\[\|\mathbb{P}_{\mu_0}\left(X_t\in \cdot|\tau_{\partial D}>t\right)-\nu_{\infty}^{\varepsilon}\|_{TV}\leqslant C_\varepsilon e^{-\chi_\varepsilon t}\frac{4\|V\|_{\infty}}{\inf_{K}\phi}.\]

For the second point, write:
\[\mathbb{P}_{\mu_0}\left(\tau_{\partial D}>t\right)\geqslant \frac{1}{4}\inf_{x\in K}\mathbb P_x(\tau_{\partial D}>t).\]
Now fix $0<\tilde U_1<\tilde U_0$ such that $K\subset F\cup \left\{U\leqslant \tilde U_0\right\}=:\tilde F$, and \[\tilde \zeta := \inf_{(F\cup \left\{U\leqslant \tilde U_1\right\})^c}|\na U|>0\].
Fix $T>0$ such that $\tilde U_0-\tilde\zeta^2 T< \tilde U_1$, and some $\delta >0$ such that:
\[\delta < \min\left(dist \left(F\cup \left\{U\leqslant \tilde U_1\right\},(F\cup \left\{U\leqslant \tilde U_0\right\})^c \right), dist\left(F\cup \left\{U\leqslant \tilde U_0\right\},\R^d\setminus D\right) \right).\]

Write:
\[\mathcal E = \left\{ \tau_{\partial D}>T, X_T \in \tilde F \right\},\]
and recall the definition~\eqref{def:flow} of $\varphi$.
With our choice of $T$, $\tilde U_0$ and $\tilde U_1$, we have that $\varphi_T \in F\cup \left\{U\leqslant \tilde U_1\right\}$ for all $x\in \tilde F$.
We get using Gronwall's Lemma:
\[\sup_{0\leqslant t\leqslant T}|X_t-\varphi_t| \leqslant e^{\|\na^2U\|_{\infty}T}\sqrt{2\varepsilon}\sup_{0\leqslant t\leqslant T}|B_t|.\]
If $W$ is a one-dimensional Brownian motion, $\sup_{0\leqslant t\leqslant T}W_t$ has the law of $|G|$ where $G$ is a standard Normal random variable. Hence:
\[\mathbb P\left(\sup_{0\leqslant t\leqslant T}|X_t-\varphi_t|\geqslant \delta \right) \leqslant 4d\mathbb P\left(G\geqslant \frac{\delta e^{-\|\na^2U\|_{\infty}T}}{T\sqrt{2\varepsilon}} \right) \leqslant 4de^{-\frac{b}{\varepsilon}}, \]where $b=\frac{\delta e^{-\|\na^2U\|_{\infty}T}}{2\sqrt{2}T}$.

Now write
\[\mathcal E_i = \left\{\tau_{\partial D} >(i+1)T \text{ and }X_{(i+1)T} \in \tilde F \right\}.\]
We showed that for $\varepsilon$ small enough, $\mathbb P\left( \mathcal E_{i+1}|\mathcal E_{i}\right) \geqslant 1-e^{-b/\varepsilon}$. We also have for our choice of $\delta$:\[\left\{ \tau_{\partial D}< t_\varepsilon \right\} \subset \mathcal{E}_{\left\lceil t/T\right\rceil}^c.\]Hence for all $x\in\tilde F$:
\[\mathbb  P_{x}\left( \tau_{\partial D} >t  \right) \geqslant 
\left(1-e^{-b/\varepsilon}\right)^{t/T},\] and thus the result.
\end{proof}

In our metastable setting, this already yields propagation of chaos at equilibrium. Indeed, if the FV process starts from its stationary measure, its law won't change. But then from Proposition~\ref{propag}, the empirical measure of this process is close to the law of the process \eqref{descente} conditioned on survival at time $t\geqslant0$ starting from $\nu_{\infty}^{N,\varepsilon}$, which is itself close to the $QSD$ if $t$ is large enough.

\begin{lem}\label{propeq}
Under Assumption~\ref{assu1} and~\ref{assu2}, there exists $\varepsilon_0>0$ such that for all $0<\varepsilon<\varepsilon_0$, there exists $C_\varepsilon,\eta_{\varepsilon,1}>0$ such that if $\textbf{X}_{\infty}$ is a random vector of law $\nu^{N,\varepsilon}_{\infty}$ on $D^{N}$, then for all bounded function $f:D\to\R_+$:
\[\E\left(\left|\int_D f \dd\pi^N(\textbf{X}_{\infty}) - \int_D f\dd\nu_{\infty}^{\varepsilon}\right|\right)\leqslant \frac{C_\varepsilon\|f\|_{\infty}}{N^{\eta_{\varepsilon,1}}}.\]
\end{lem}

\begin{proof}
Assumptions~\ref{assu1} and~\ref{assu2} yield the existence of $\nu_{\infty}^{N,\varepsilon}$. Introduce the FV process $(\textbf{X}_t)_{t\geqslant0}$ with initial condition $\textbf X_{\infty}$. By definition, the law of $\textbf{X}_t$ is $\nu_\infty^{N,\varepsilon}$ for all $t\geqslant 0$. Recall the definition of $\mathfrak B$ and $A$ from \eqref{numbbound}, and set $K=D\setminus \mathfrak B$. Since $\nu^{N,\varepsilon}_\infty$ is the stationary measure of the FV process, we get from Lemma~\ref{boundary} applied with $\alpha=1/2$ that for all $t\geqslant 0$:
\[\mathbb{P}\left(A\left(\textbf{X}_{t}\right)>N/2\right)\leqslant q_{\varepsilon}^N,\]where $q_\varepsilon>0$ goes to zero as $\varepsilon$ goes to zero. Recall the definition of $\lambda_{\varepsilon}$ from Proposition~\ref{convnonlin}, let $t=b\ln(N)$ for some $0<b<1/(2\lambda_{\varepsilon})$, and write $\mathcal{A}=\left\{ G_1\left(\textbf{X}_{t}\right) >N/2\right\}$. We have that:
\[\E\left(\left|\int_D f \dd\pi^N(\textbf{X}_{\infty}) - \int_D f\dd\nu_{\infty}^{\varepsilon}\right|\right)=\E\left(\left|\int_D f \dd\pi^N(\textbf{X}_{t}) - \int_D f\dd\nu_{\infty}^{\varepsilon}\right|\right),\]
and
\begin{multline*}
\E\left(\left|\int_D f \dd\pi^N(\textbf{X}_{\infty}) - \int_D f\dd\nu_{\infty}^{\varepsilon}\right|\right)=\E\left(\left|\int_D f \dd\pi^N(\textbf{X}_{t}) - \int_D f\dd\nu_{\infty}^{\varepsilon}\right|\right) \\ \leqslant  \E\left(\left|\int_D f \dd\pi^N(\textbf{X}_{t})-\E_{\pi^N(\textbf{X}_{\infty})}(f(X_t)|\tau_{\partial D}>t)\right|\right) \\+ \E\left(\left|\E_{\pi^N(\textbf{X}_{\infty})}(f(X_t)|\tau_{\partial D}>t)-\int_D f \dd\nu_{\infty}^{\varepsilon}\right|\right).
\end{multline*}
On $\mathcal G_1^c$, we have that $\pi^N(\textbf{X}_{\infty})(K)\geqslant 1/2$. Hence, from Proposition~\ref{propag} and \ref{convnonlin}, there exists $C>0$ such that:
\begin{equation*}
\E\left( \left| \int_D f\dd\pi^N(\textbf{X}_t) - \E_{\pi^N(\textbf{X}_{\infty})}(f(X_t)|\tau_{\partial D}>t)\right|\right) \leqslant \frac{C\|f\|_{\infty}}{\sqrt{N}}\frac{1}{N^{-b\lambda_\varepsilon}-q_\varepsilon^N} =\frac{C\|f\|_{\infty}}{N^{1/2-b\lambda_\varepsilon}}.
\end{equation*} 
From Proposition~\ref{convnonlin}, we get that:
\[\E\left(\left|\E_{\pi^N(\textbf{X}_{\infty})}(f(X_t)|\tau_{\partial D}>t)-\int_Df \dd\nu_{\infty}^{\varepsilon}\right|\right)\leqslant \left(Ce^{-\chi_{\varepsilon}t}+q_\varepsilon^N\right)\|f\|_{\infty} = \frac{C\|f\|_{\infty}}{N^{b \chi_\varepsilon}}.\]
The fact that we chose $b<1/2\lambda_\varepsilon$ concludes, as soon as $\varepsilon$ is small enough so that $q_\varepsilon<1$.
\end{proof}

\begin{proof}[Proof of Theorem~\ref{thm2}]
Fix some $\varepsilon>0$, some compact $K\subset D$, $\mu_0\in\mathcal M^1(D)$ such that $\mu_0(K)\geqslant 1/2$, and a random variable $\textbf X_0$ of law $\mu_0^{\otimes N}$. Write:
\[\mathcal G_2 = \left\{\pi^N(\textbf X_0)(K) \geqslant 1/4 \right\}.\]
We have that $\E(\pi^N(\textbf X_0)(K)) = \mu_0(K)$, and $\text{Var} (\pi^N(\textbf X_0)(K)) = \mu_0(K)(1-\mu_0(K))/N$. Hence we have:
\[\mathbb P(\mathcal G_2^c) \leqslant 4/N.\]
We now fix $0<b<1/(2\lambda_\varepsilon)$. For all $t\leqslant b\ln(N)$, from Propositions~\ref{propag} and \ref{convnonlin}, we get that there exists $C>0$ such that:
\begin{align*}
\E\left( \left|\int_D f\dd \pi^N(\textbf{X}_t) - \E_{\pi^N(\textbf X_0)}\left(f\left(X_t\right)\middle|\tau_{\partial D} >t\right) \right| \right) \leqslant \frac{C\|f\|_{\infty}}{\sqrt{N}}\sqrt{\frac{1}{N^{-b\lambda_\varepsilon}-CN^{-1}}}\leqslant \frac{C\|f\|_{\infty}}{N^{1/2-b\lambda_\varepsilon/2}}.
\end{align*}
If $t\geqslant b\ln(N)$, we consider a random variable $(\textbf X_t,\textbf{X}_\infty)$ which is an optimal coupling of $\mu_0^{\otimes N}P_t^{N,\varepsilon}$ and $\nu_{\infty}^{N,\varepsilon}$ for the distance $\mathbf d$ defined in~\eqref{distance}. We then bound as follow:
\begin{align*}
\E\left( \left|\int_D f\dd \pi^N(\textbf{X}_t) - \E_{\pi^N(\textbf X_0)}\left(f\left(X_t\right)\middle|\tau_{\partial D} >t\right) \right| \right) &\leqslant \E\left( \left|\int_D f\dd \pi^N(\textbf{X}_t) -\int_D f\dd \pi^N(\textbf{X}_\infty) \right| \right) \\ &+ \E\left( \left|\int_D f\dd \pi^N(\textbf{X}_\infty) - \int_D f \dd\nu_\infty^{\varepsilon} \right| \right) \\ &+ \E\left( \left|\int_D f \dd\nu_\infty - \E_{\pi^N(\textbf X_0)}\left(f\left(X_t\right)\middle|\tau_{\partial D} >t\right) \right| \right).
\end{align*}
From the proof of Theorem~\ref{thm}, the first term is bounded by:
\[\frac{\E(d(\textbf X_t,\textbf{X}_\infty))}{N}\|f\|_{\infty}\leqslant C(1-c)^{t/t_\varepsilon}\|f\|_{\infty}\leqslant CN^{\ln(1-c)/t_\varepsilon}\|f\|_{\infty}.\]
From Lemma~\ref{propeq}, the second term is bounded by:
\[\frac{C}{N^{\eta_{\varepsilon,1}}}\|f\|_{\infty}.\]
Finally, by Proposition~\ref{convnonlin}, the third term is bounded by:
\[\left(Ce^{-\chi_\varepsilon t}+2\mathbb P(\mathcal G_2)\right)\|f\|_{\infty}\leqslant CN^{-\chi_\varepsilon b}\|f\|_{\infty},\]and thus the result.
\end{proof}

\section{Establishing  Assumption~\ref{assu2}}\label{technique}

This section is devoted to the proof of the following:

\begin{lem}\label{unifexitevent}
Under Assumptions~\ref{assu1} and \ref{assu3}, Assumption~\ref{assu2} holds.
\end{lem}

\begin{proof}
\begin{enumerate}
    \item Suppose first that $d=1$. Under Assumption~\ref{assu1}, $D=(x_1,x_2)$ for some $x_1<x_2$. Set $\mathcal N = (x_1+\theta,x_2-\theta)$ with $\theta>0$ sufficiently small so that $U'\neq 0$ on $(x_1,x_1+\theta)\cup (x_2-\theta,x_2)$ (as a consequence of Assumption~\ref{assu1}, $U$ is thus   decreasing (resp. increasing) on $(x_1,x_1+\theta)$ (resp. $(x_2-\theta,x_2)$)) and so that $\min(U(x_1+\theta),U(x_2-\theta))>c^*$ (which is possible since $\min(U(x_1),U(x_2))=U_0>c^*$). Take any $a\in (c^*,\min(U(x_1+\theta),U(x_2-\theta)))$. Fix $T_0>0$ such that $\varphi_{T_0}(x_1) >x_1 + \theta$ and $\varphi_{T_0}(x_2) < x_2 -\theta$, which is possible since $U'\neq 0$ on $D\setminus \mathcal N$. Then, $\mathcal N$ is a neighborhood of $\varphi_{T_0}(D)$, and using the uniqueness of the solution to equation~\eqref{descente} we have:
    \[\left\{\exists x\in D\setminus \mathfrak B_1, \tau_{\partial D}(X^x) <t_\varepsilon \right\} = \left\{ \tau_{\partial D}(X^{x_1+\theta}) <t_\varepsilon \right\}\cup \left\{ \tau_{\partial D}(X^{x_2-\theta}) <t_\varepsilon \right\}.\]
    We can conclude with  \cite[Chapter 6, Theorem 6.2]{Freidlin-Wentzell}.  
    
    \item We now suppose the second condition of Assumption~\ref{assu3}.  For $T>0$, let 
    \[\mathfrak B(T) = \{\varphi_t(x),\ t\in[0,T],\ x\in\partial D\}.\]
    We have that 
    \[
    \lim_{T\rightarrow\infty} \sup_{x\in\partial D} U\po \varphi_T(x)\pf = U_c,
    \]
    hence we may fix $T_0>1$ large enough so that 
    \[
    c^* < (U_0-U_1)/2,\qquad U_1 := \sup_{x\in\partial D} U\po \varphi_{T_0-1}(x)\pf.
    \]
    Set 
    \[\mathcal N = D\setminus\po \left\{ U > (U_0 + U_1)/2 \right\}\cap \mathfrak B(T_0-1)\pf\,,\]
    and fix any $a\in (c^*,(U_0-U_1)/2)$. Because $U_1 < (U_1+U_0)/2$,  this amounts to say that $\mathcal N$ is the connected component of $\{U>(U_0+U_1)/2\}$ in $D$ whose closure contains $\partial D$, and in particular $\mathcal N$ is an open set (see an illustration in Figure~\ref{Figure2} to fix ideas). Moreover, for all $x\in \overline{D}$, $\varphi_{T_0}(x) \notin \mathfrak B(T_0-1)$ (by uniqueness of solutions of ODE) so that $\mathcal N$ is a neighborhood of $\varphi_{T_0}(\overline{D})$ (as required in Assumption~\ref{assu2}).
    Fix some 
    \begin{align*}
    &U_2\in \left( a + (U_0+U_1)/2,U_0 \right), \\ &U_3\in \left(U_1+a,(U_0+U_1)/2\right),\\ &0<\gamma<\min(U_3-U_1-a,U_2-(U_0+U_1)/2-a),
    \end{align*}
    see Figure~\ref{Figure1}.

    \begin{figure}
    \centering
    \includegraphics[scale=0.5]{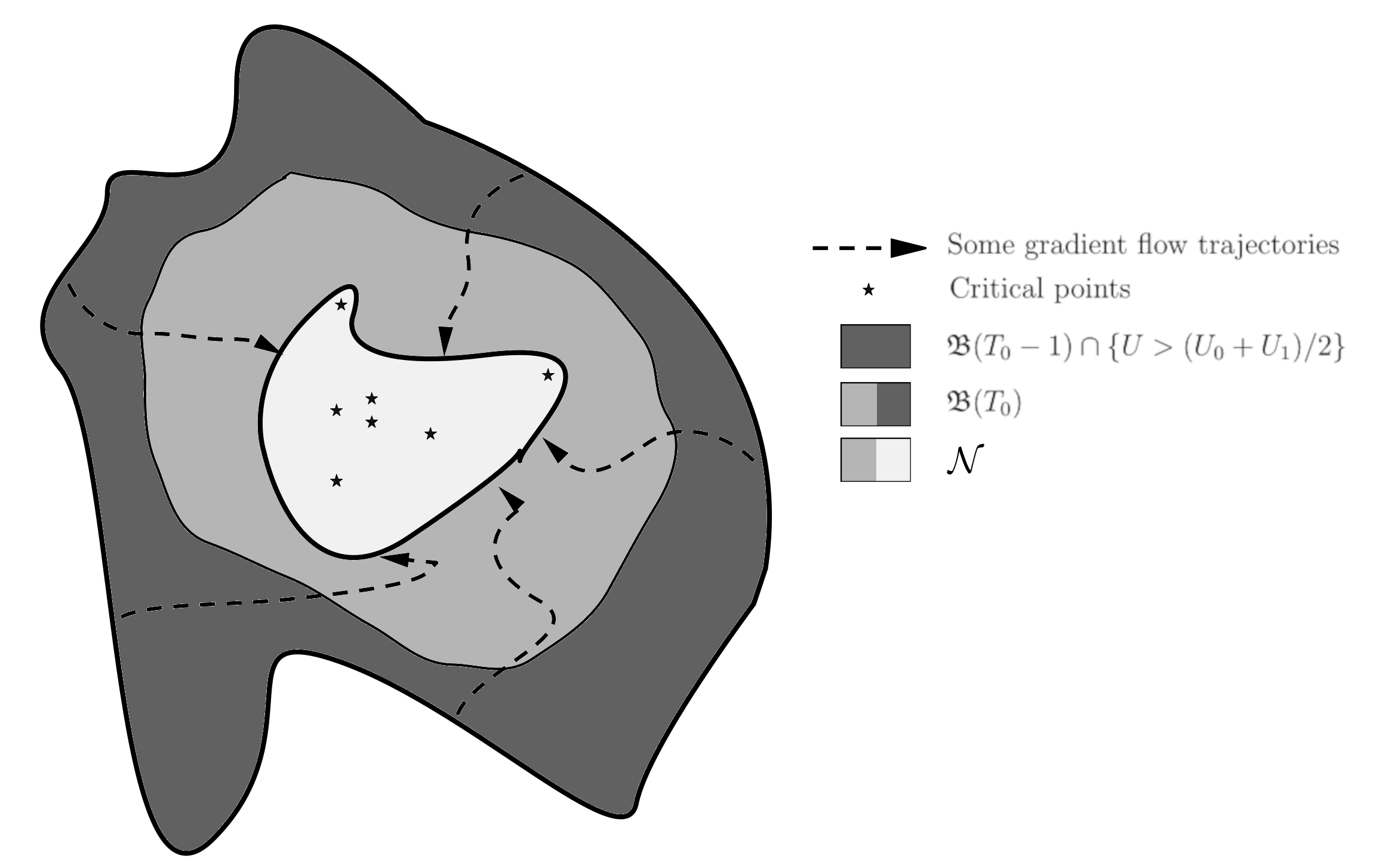}
    \caption{Sketch of the sets introduced in the proof of Lemma~\ref{unifexitevent}. In the lightest area, in the center of $D$, $U$ may have critical points and take high values, however at the boundary of this domain (which is the image of $\partial D$ by the gradient flow $\varphi_{T_0}$) the energy is relatively low (less than $U_1$). By contrast, the darkest area is a neighborhood of $\partial D$ where the energy is relatively high (larger than $(U_0+U_1)/2>U_1$).}
    \label{Figure2}
    \end{figure}
    
    \begin{figure}
    \centering
    \includegraphics[scale=0.6]{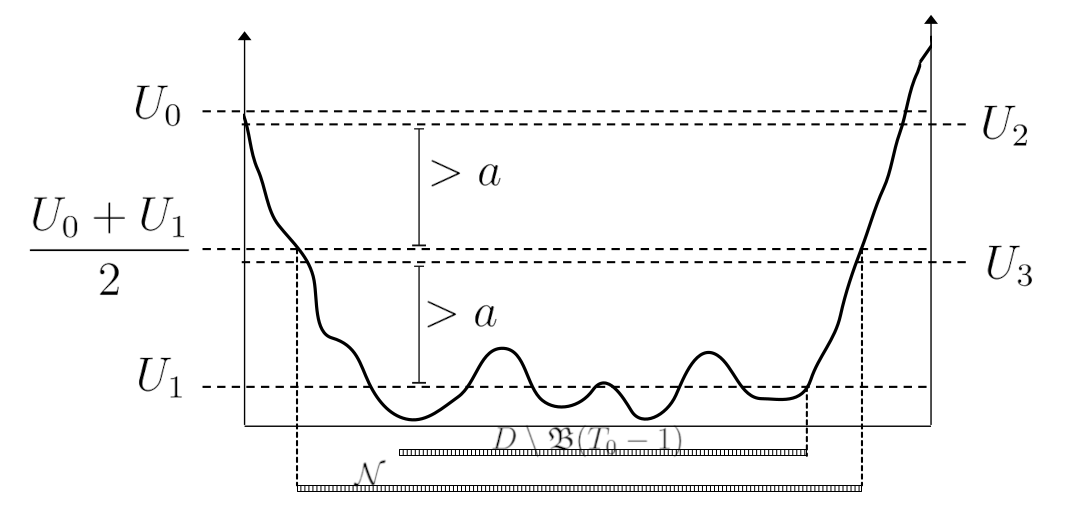}
    \caption{First, $U_0$ is the  minimal energy level on $\partial D$, and $U_1$ is the maximal energy level on $\varphi_{T_0-1}(\partial D)$ (hence at the boundary of $D\setminus \mathfrak B(T_0-1)$). The intermediate level $(U_0+U_1)/2$ is used to define $\mathcal N$. Then $U_2$ and $U_3$ are taken slightly below respectively $U_0$ and $(U_0+U_1)/2$ in such a way that there is still a gap larger than $a>c_*$ between $U_2$ and $(U_0+U_1)/2$ and between $U_3$ and $U_1$. The condition that $c^*<(U_0-U_1)/2$ means that it is possible to find $U_2,U_3$ and $a$ to fulfill these constraints.}
    \label{Figure1}
    \end{figure}
    
    Fix as well some $T>0$ (to be chosen large enough below, independently from $\varepsilon$) and for all $x\in \mathcal N$ write the event:
    \[
    \mathcal E^x = \left(\left\{U(X^x_T)<U_3 \right\} \cup \left\{ X^x_T\notin \mathcal B(T_0) \right\} \right) \cap \left\{ \forall 0<t<T, X_t^x\notin \left\{U<U_2\right\} \cap \mathcal B(T_0) \right\}.
    \]
    Let's first show that there exists $C>0$ such that $\mathbb{P}((\mathcal E^x)^c) \leqslant Ce^{-(a+\gamma)/\varepsilon}$ for $\varepsilon$ small enough and $T$ great enough, using Large Deviation results. In other words, in a fixed time interval, with high probability, starting from either a medium energy level at most $(U_1+U_0)/2$, or from the center of the domain, the process will stay away from the neighborhood of $\partial D$ where the energy is above $U_2$, and will end up either below the medium energy level $U_3$, or at the center of the domain. According to \cite[Chapter 4, Theorem 1.1]{Freidlin-Wentzell}, the action function of the process~\eqref{descente} is:
    \begin{eqnarray}
    I(\Phi) &=& \frac{1}{4\varepsilon}\int_0^T |\Phi'_s + \na U(\Phi_s)|^2\dd s \nonumber\\
    &= & \frac{1}{4\varepsilon}\int_0^T \left(|\Phi_s'|^2 + |\na U(\Phi_s)|^2\right)\dd s + \frac{U(\Phi_T)-U(\Phi_0)}{2\varepsilon}\label{Iphi1}\\
    & =& \frac{1}{4\varepsilon}\int_0^T |\Phi'_s - \na U(\Phi_s)|^2\dd s + \frac{U(\Phi_T)-U(\Phi_0)}{\varepsilon}.\label{Iphi2}
    \end{eqnarray}
    For all function $\Phi:[0,T]\to D$ such that $\Phi_0\in\mathcal N$ and there exists $t\in (0,T)$ such that $U(\Phi_t)\geqslant U_2$ and $\Phi_t\in \mathcal{B}(T_0)$, there exists $0\leqslant u < t$ such that $\Phi_u \in \po\partial \mathcal B(T_0) \pf\cap D$, and hence $U(\Phi_u) <U_1 < (U_0+ U_1)/2$, and using \eqref{Iphi2}, we have:
    \[
    \varepsilon I(\Phi)  \geqslant \frac{1}{4}\int_u^t |\Phi'_s + \na U(\Phi_s)|^2\dd s \geqslant U(\Phi_t) - U(\Phi_u) > U_2-(U_0+U_1)/2 >a+\gamma.
    \]
    From this we deduce that for $\varepsilon$ small enough, for all $x\in\mathcal N$:
    \[
    \mathbb P\left(\exists t\in [0,T], X_t^x\notin \left\{U<U_2\right\} \cap \mathcal B(T_0) \right)\leqslant e^{-(a+\gamma)/\varepsilon}.
    \]
    Second, to bound the probability that $X_T^x \in \left\{U\geqslant U_3 \right\} \cap \mathcal B(T_0)$, we consider two possible events: either the process stays during the whole interval $[0,T]$ in $\left\{ U\geqslant U_1 \right\} \cap \mathcal B(T_0)$ (which is unlikely because it would mean it stays in an unstable region where $\na U$ is non-zero), 
    or the energy of the process goes down to $U_1$ but then climbs back in a time less than $T$ to the level $U_3$ (which is also unlikely).
    More precisely, notice that $\mathfrak B(T_0)$ is a compact set where $\na U$ is non-zero, so that $|\na U|$ is uniformly bounded from below on $\left\{U\geqslant U_1\right\} \cap \mathcal B(T_0)$, and for all functions $\Phi:[0,T]\to D$ such that $\Phi_t\in \left\{U\geqslant U_1 \right\} \cap \mathcal B(T_0)$ for all $t\in[0,T]$, using \eqref{Iphi1}, we have that 
    \[\varepsilon I(\Phi) \geqslant T \inf_{\left\{U\geqslant U_1 \right\} \cap \mathcal B(T_0)}|\na U|^2/4 - U_3/2.\]
    We chose $T$ great enough so that for all such functions, $\varepsilon I(\Phi)> a+\gamma$. Then for $\varepsilon$ small enough we have that for all $x\in \left\{ U \leqslant (U_0 + U_1)/2 \right\} \cap \mathcal B(T_0)$,
    \[\mathbb P\left( \exists t\in[0,T], X^x_t\in \left\{U <U_1 \right\} \cup \mathcal B(T_0)^c \right) \geqslant 1-e^{-(a+\gamma)/\varepsilon}.\]
    Next, for all functions $\Phi:[0,T]\to D$ such that $\Phi_0\in \left\{U <U_1 \right\} \cup \mathcal B(T_0)^c$ and $\Phi_T\in\left\{U>U_3\right\}\cap \mathcal B(T_0)$, 
    \[\varepsilon I(\Phi)>U_3-U_1.\]
    Indeed, if $\Phi_0\in \mathcal B(T_0)^c$ and $\Phi_T\in \mathcal B(T_0)$, there exists $0<t<T$ such that $\Phi_t \in \partial \mathcal B(T_0)\cap D$ and $\sup_{\partial \mathcal B(T_0) \cap D} U < U_1$ so that $U(\Phi_t) <U_1$.
    Hence, for all $x$ such that $x\in \left\{U <U_1 \right\} \cup \mathcal B(T_0)^c$ and $\varepsilon$ small enough :
    \[\mathbb P(\exists t\in [0,T], U(X^x_t)> U_3) \leqslant e^{-(a+\gamma)/\varepsilon}.\]
    From those last two bounds we get for all $x\in\mathcal N$:
    \begin{align*}
        &\mathbb P\left( U(X^x_T) > U_3, X_T^x \in \mathcal B(T_0) \right) \\ &\qquad \leqslant \mathbb P\left( \exists T>s>t>0, X^x_t\in \left\{ U < U_1 \right\} \cup \mathcal B(T_0)^c, U(X^x_T)>U_3 \right)  \\&\qquad \qquad \qquad \qquad \qquad \qquad \qquad \qquad  + \mathbb{P} \left( \forall t<T, X^x_t \in \left\{ U\geqslant U_1 \right\} \cap \mathcal B(T_0) \right) \\ 
        &\qquad  \leqslant 2e^{-(a+\gamma)/\varepsilon}.
    \end{align*}
    Finally we get for all $x\in\mathcal N$:
    \[\mathbb P\left( (\mathcal E^x)^c\right) \leqslant 3e^{-(a+\gamma)/\varepsilon}.\]
    
    Up to now, we only have a control for a fixed initial condition $x$. To tackle simultaneously all initial conditions in $\mathcal N$, we use that, in a time $T$, two processes which start close stay close (deterministically). More precisely, fix 
    \[
    \delta 
    <
    \min\left(\mathrm{dist}(\left\{ U\leqslant U_2\right\}\cup \mathcal B(T_0),\R^d\setminus D),\mathrm{dist}(\left\{ U\leqslant U_3\right\} \cup \mathcal B(T_0)^c ,D\setminus \mathcal N)\right),
    \]
    and $\delta'>0$ such that $\delta'e^{\|\na^2U\|_{\infty} T}<\delta$. Fix a family of point $z_1,\dots,z_k \in\mathcal N$ such that $\mathcal N\subset \cup_{i=1}^k B(z_i,\delta')$, where $B(z,r)$ is the ball of center $z$ and radius $r$. Write the event:
    \[\mathcal E=\left\{ \forall x\in\mathcal N, \tau_{\partial D}(X^x)>T \text{ and }X^x_T\in\mathcal N \right\}.\]
    If $x\in\mathcal N$, there exists $i$ such that $|x-z_i|<\delta'$. Gronwall's lemma then classically yields that \[\sup_{0\leqslant t\leqslant T} |X^x_t-X^{z_i}_t| \leqslant \delta'e^{\|\na^2U\|_{\infty}T} <\delta.\] 
    In particular, $\tau_{\partial D}(X^x) < T$ implies that $U(X_t^{z_i}) \geqslant U_2$ for some $t\in[0,T]$. 
    Hence we have that:
    \[  \bigcap_{i=1}^k\mathcal E^{z_i}\subset \mathcal E,\]and for $\varepsilon$ small enough \[\mathbb P(\mathcal E^c)\leqslant 3ke^{-(a+\gamma)/\varepsilon}.\]
Now write
\[\mathcal E_i = \left\{\forall x\in\mathcal N,\tau_{\partial D}(X^x) >(i+1)T \text{ and }X^x_{(i+1)T} \in\mathcal N \right\}.\]
We showed that for $\varepsilon$ small enough, $\mathbb P\left( \mathcal E_{i+1}|\mathcal E_{i}\right) \geqslant 1-3ke^{-(a+\gamma)/\varepsilon}$. We also have that \[\left\{\exists x\in\mathcal N, \tau_{\partial D}(X^x)< t_\varepsilon \right\} \subset \mathcal{E}_{\left\lfloor t_\varepsilon/T\right\rfloor}^c.\]Hence:
\[\mathbb  P\left( \forall x\in\mathcal N, \tau_{\partial D}(X^x) >t_\varepsilon  \right) \geqslant 
\left(1-3ke^{-(a+\gamma)/\varepsilon}\right)^{e^{a/\varepsilon}/T},\] and thus this probability goes to $1$ as $\varepsilon$ goes to $0$. 
\end{enumerate}    
\end{proof}

\section*{Acknowledgements}

This work is supported by the projects SWIDIMS (ANR-20-CE40-0022) and CONVIVIALITY (ANR-23-CE40-0003) of the French National Research Agency.

\bibliographystyle{plain}
\bibliography{biblio}

\end{document}